\documentclass[leqno,letterpaper]{article}
\pdfoutput=1
\usepackage[margin=1in]{geometry}

\usepackage[utf8]{inputenc}
\usepackage[T1]{fontenc}
\usepackage{csquotes}

\DeclareUnicodeCharacter{012D}{\u\i}

\usepackage[style=alphabetic,sorting=nyt,hyperref=auto,backref=false,backend=biber,bibencoding=auto]{biblatex}
\bibliography{gen-u}%

\usepackage[utopia]{mathdesign}
\usepackage{textcomp}

\usepackage{
  amsthm,%
  amsxtra%
}
\usepackage{mathtools}
\usepackage{dsfont}

\usepackage[all,tips,2cell,pdf]{xy}
\SelectTips{cm}{}
\UseAllTwocells

\usepackage{xspace} 
\usepackage{ifpdf}
\usepackage{ifthen}
\usepackage{xcolor}

\usepackage{microtype} 

\ifthenelse{\boolean{pdf}}{%
  \definecolor{theblue}{rgb}{0.02,0.04,0.7}%
  \definecolor{thered}{rgb}{0.8,0.04,0.07}%
  \definecolor{thegreen}{rgb}{0.06,0.44,0.08}%
  \definecolor{thegrey}{gray}{0.5}%
  \definecolor{theshade}{gray}{0.92}%
  \usepackage[%
  colorlinks=true,%
  citecolor=thegreen,%
  linkcolor=theblue,%
  urlcolor=thered,%
  plainpages=false,%
  pdfpagelabels,bookmarks]{hyperref}}{%
  \usepackage[hypertex]{hyperref}%
}

\swapnumbers
\newtheorem*{theorem*}{Theorem}
\newtheorem{theorem}[equation]{Theorem}
\newtheorem{proposition}[equation]{Proposition}
\newtheorem*{proposition*}{Proposition}
\newtheorem{lemma}[equation]{Lemma}

\newtheorem{lemma-def}[equation]{Lemma-Definition}
\newtheorem{claim}[equation]{Claim}

\theoremstyle{definition}
\newtheorem{definition}[equation]{Definition}
\newtheorem{example}[equation]{Example}
\newtheorem*{example*}{Example}
\newtheorem{remark}[equation]{Remark}
\newtheorem*{remark*}{Remark}

\newtheorem*{variant*}{Variant}

\newtheorem*{notation*}{Notation}
\newtheorem{paragr}[equation]{} 

\numberwithin{equation}{subsection}%
\newcommand{\cprod}[1]{\wedge^{#1}} 

\newcommand{\eqdef}{\overset{\text{\normalfont\tiny def}}{=}}
\renewcommand{\phi}{\varphi}
\renewcommand{\epsilon}{\varepsilon}
\newcommand{\ab}{\mathrm{ab}}
\newcommand{\op}{\mathrm{op}}
\newcommand{\pt}{\mathrm{pt}}

\newcommand{\loccit}{loc.\ cit.\xspace}

\newcommand{\ie}{i.e.\xspace}
\newcommand{\eg}{e.g.\xspace}
\newcommand{\W}{\overline{\mathrm{W}}}

\newcommand{\lto}{\longrightarrow} 
\newcommand{\lmto}{\longmapsto}
\newcommand{\iso}{\simeq}
\newcommand{\isoto}{\overset{\sim}{\rightarrow}} 
\newcommand{\lisoto}{\overset{\sim}{\longrightarrow}} 

\newcommand{\field}[1]{\ensuremath{\mathds{#1}}}
\newcommand{\ZZ}{\field{Z}}

\newcommand{\derived}[1]{\mathds{#1}}
\newcommand{\derL}{\derived{L}}
\newcommand{\derD}{\derived{D}} 

\newcommand{\HH}{\mathrm{HH}}
\newcommand{\SH}{\mathrm{SH}}

\newcommand{\DerD}{\mathrm{D}} 

\newcommand{\cat}[1]{\mathsf{#1}} 
\newcommand{\twocat}[1]{\mathbf{#1}}  
\newcommand{\bicat}[1]{\mathbf{#1}} 
\DeclareMathOperator{\Ob}{Ob}

\newcommand{\B}{\cat{B}}
\newcommand{\C}{\cat{C}}
\newcommand{\D}{\cat{D}}
\newcommand{\E}{\cat{E}}
\newcommand{\G}{\cat{G}}

\newcommand{\kalg}{k\text{-}\cat{Alg}}

\newcommand{\site}[1]{\cat{#1}}

\newcommand{\sS}{\site{S}}
\newcommand{\s}{\sS}

\newcommand{\topos}[1]{\cat{#1}}
\newcommand{\T}{\topos{T}}

\newcommand{\shv}[1]{\site{#1}\sptilde}
\newcommand{\shs}{\shv{S}}

\newcommand{\smp}[1]{\underline{#1}}

\newcommand{\stack}[1]{\mathscr{#1}}
\newcommand{\sta}{\stack{A}}
\newcommand{\stb}{\stack{B}}

\newcommand{\stp}{\stack{P}}
\newcommand{\stq}{\stack{Q}}
\newcommand{\str}{\stack{R}}
\newcommand{\sts}{\stack{S}}
\newcommand{\stt}{\stack{T}}

\newcommand{\twostack}[1]{\mathfrak{#1}}
\newcommand{\bistack}[1]{\mathfrak{#1}}

\newcommand{\xbm}{\bicat{XBiMod}} 
\newcommand{\XBM}{\bistack{XBiMod}} 
\newcommand{\ch}{\bicat{Ch}^{[-1,0]}}
\newcommand{\CH}{\bistack{Ch}^{[-1,0]}}
\newcommand{\tors}{\operatorname{\textsc{Tors}}}

\newcommand{\picstacks}{\twostack{Pic}}
\newcommand{\PIC}{\picstacks}
\newcommand{\pic}{\twocat{Pic}}

\DeclareMathOperator{\coker}{coker}

\newcommand{\del}{\partial} 

\DeclareMathOperator{\Der}{Der}
\DeclareMathOperator{\shDer}{\underline{\mathrm{Der}}}
\DeclareMathOperator{\shDerD}{\underline{\DerD}}

\DeclareMathOperator{\shcatEnd}{\mathscr{E}\mspace{-3mu}\mathit{nd}}
\DeclareMathOperator{\Ext}{Ext}
\DeclareMathOperator{\catExt}{\cat{Ext}}
\DeclareMathOperator{\shcatExt}{\mathscr{E}\mspace{-3mu}\mathit{xt}}

\DeclareMathOperator{\catExtAlg}{\cat{ExtAlg}}
\DeclareMathOperator{\shcatExtAlg}{\mathscr{E}\mspace{-3mu}\mathit{xt}\mathscr{A}\mspace{-2mu}\mathit{lg}}

\DeclareMathOperator{\Hom}{Hom} 
\renewcommand{\hom}{\Hom}
\DeclareMathOperator{\catHom}{\cat{Hom}}

\DeclareMathOperator{\shHom}{\underline{\mathrm{Hom}}}
\DeclareMathOperator{\shcatHom}{\mathscr{H}\mspace{-3mu}\mathit{om}}
\DeclareMathOperator{\id}{id} 
\DeclareMathOperator{\im}{Im}

\DeclareMathOperator{\N}{N}

\author{Ettore Aldrovandi\footnote{\url{aldrovandi@math.fsu.edu}}\\
  \small Department of Mathematics, Florida State University}

\title{Stacks of $\mathit{Ann}$-Categories and their morphisms}%
\date{}

\begin{document}
\maketitle
\begin{abstract}
  We show that $\mathit{ann}$-categories admit a presentation by crossed bimodules, and prove that morphisms between them can be expressed by special kinds spans between the presentations. More precisely, we prove the groupoid of morphisms between two $\mathit{ann}$-categories is equivalent to that of bimodule butterflies between the presentations.  A bimodule butterfly is a specialization of a butterfly, i.e. a special kind of span or fraction, between the underlying complexes
\end{abstract}

\phantomsection  
\addcontentsline{toc}{section}{Introduction}  
\section*{Introduction}
A categorical ring is a category carrying a bimonoidal structure which resembles that of a ring, up to natural isomorphisms and coherence conditions.  There are different notions of categorical ring, according to the strength of the commutativity axiom imposed on the underlying additive categorical group. Usually the underlying categorical group is assumed to be a symmetric one \cite{MR2369166,MR3085798}. This is a sort of ``fixed point:'' in a companion paper \cite{biext2015} we have (among other things) explored the possibility of relaxing the commutativity of the additive structure, by assuming just a braiding. However, in the unital case, a categorical ring satisfying these more relaxed axioms turns out to be equivalent to one of the usual sort.

Here we take a different approach, and we explore the case where the commutativity  law on the additive structure is actually stricter, namely we study categorical rings whose underlying categorical groups are actually strictly Picard groupoids. These where introduced, under the name $\mathit{Ann}$-categories, in a series of works~\cite{MR2065328,MR2458668}; the term \emph{regular} $\mathit{Ann}$-category is used in \cite{MR3085798}. In parallel with the classical analysis of $\mathit{gr}$-categories (i.e.\ categorical groups) carried out in \cite{sinh:gr-cats}, $\mathit{Ann}$-categories were found to be classified by the  third Shukla cohomology of rings. (By contrast, categorical rings of the more general breed discussed above are classified by the third Mac~Lane cohomology: degree three is the level at which the two theories begin to diverge, although precise comparisons exist \cite{MR0098773,MR2270566}.)

It is well known the third Shukla cohomology occurs in the classification of 2-term extensions of the form
\begin{equation*}
  \xymatrix@1{%
    0 \ar[r] & A\ar[r] & M \ar[r]^\del & R \ar[r] & B \ar[r] & 0},
\end{equation*}
where $B$, $R$ are rings (or $k$-algebras, fixing a commutative ring $k$), $M$ an $R$-bimodule and $A$ a $B$-bimodule. These can be taken to be objects in a topos $\T$, which we assume to be of the form $\operatorname{Sh}(\s)$, for a site $\s$, whenever convenient. $A$, $B$, and $\del \colon M\to R$ satisfy certain axioms, discussed below, which in particular define the structure of crossed bimodule for $\del\colon M\to R$. The link with categorical groups—in fact, with $\mathit{Ann}$-categories—is that the Picard groupoid associated to the complex $\del\colon M\to R$ carries such a structure. We show below that this remains true for the stack $\str$ associated to that groupoid.

We start by requiring that our categorical rings (see below for the precise terminology adopted here) be in particular Picard groupoids.  Thus, we start from the monoidal structure on the 2-category $\pic$ of Picard stacks described by Deligne in the seminal \cite{0259.14006}, and define a categorical ring as a unital monoid object in this 2-category. We call it a ring, or ring-like stack, and in effect it is an object fibered in $\mathit{Ann}$-categories.

Our main interest is the structure of the 2-category of monoids in $\pic$, rather than the classification issue.  In general terms, our main results are that every ring-like stack is locally equivalent to the Picard groupoid associated to a crossed bimodule, and that the 2-category they form is equivalent to the bicategory $\xbm$ of crossed bimodules of $\T$. More precisely, we have, first:
\begin{theorem*}[Theorem~\ref{thm:1}]
  Let $\str$ be a ring-like stack. Then $\str$ is equivalent to the stack associated to a crossed bimodule $\del \colon M\to R$.
\end{theorem*}
Given this, there arises the question of calculating the groupoid $\catHom(\sts,\str)$ of morphisms $F\colon \sts\to \str$ of ring-like stacks in terms of the presentations that are guaranteed to exist by the theorem. Here the situation is similar to the one dealt with in the case of group-like stacks in \cite{ButterfliesI}, namely that $F$ above does not translate into a naïve morphism of crossed bimodules. The correct translation is that $F$ corresponds to a diagram of the form
\begin{equation*}
  \vcenter{\xymatrix@-1pc{%
      N \ar[dd]_\del \ar[dr]^\kappa && M \ar[dl]_\imath \ar[dd]^\del \\
      & E \ar[dl]_\pi \ar[dr]^\jmath \\
      S && R
    }}
\end{equation*}
defined in sect.~\ref{sec:bimodule-butterflies} below. This is what we call a butterfly of crossed bimodules. A salient feature is that the anti-diagonal is a ring (or $k$-algebra) extension, in general non-singular.  Fixing the ``wings,'' one shows that these diagrams form a groupoid $\B(S_\bullet,R_\bullet)$, and we have
\begin{theorem*}[Theorem~\ref{thm:2}]
  There are equivalences
  \begin{equation*}
    \catHom(\sts,\str) \iso \B(S_\bullet,R_\bullet)\quad\text{and}\quad
    \shcatHom (\sts,\str) \iso \stb (S_\bullet,R_\bullet).    
  \end{equation*}
\end{theorem*}
The objects on the right are the stack versions, obtained by restricting the theorem to a variable object $S$ of $\T$.

Define $\xbm$ as the bicategory whose objects are the crossed bimodules of $\T$ and whose $\hom$-categories are the groupoids $\B(S_\bullet,R_\bullet)$. Analogously to \cite{ButterfliesI}, there is a composition $\B(S_\bullet,R_\bullet) \times  \B(T_\bullet,S_\bullet)\to \B(T_\bullet,R_\bullet)$ defined in sect.~\ref{sec:composition}. Matters are very similar to the group-like situation, so we obtain:
\begin{theorem*}[Theorems~\ref{thm:4} and~\ref{thm:3}]
  $\xbm$ is equivalent to the 2-category of ring-like stacks, \ie the 2-category of monoids in $\pic$.

  The fibered bicategory defined by $U \to \xbm (\s/U)$, $U\in \Ob (\s)$, is a (weak) 2-stack $\XBM$ which is equivalent to the 2-stack of monoids in $\PIC$, the 2-stack of Picard stacks.
\end{theorem*}

The previous theorem implies that crossed bimodules satisfy a form of 2-descent along butterflies, just as in the group-like case. In other words, a ring-like stack $\str$ can be obtained as a bilimit, relative to a hypercover $V_\bullet$ of the terminal object of $\T$, of a diagram of crossed bimodules where the various morphisms in the descent conditions are expressed by butterflies. We use this to show, in the last section, that a ring-like stack $\str$ determines a class in (a sheafified form) of Shukla cohomology. A posteriori this justifies the fact that the categorical rings considered here are of the same type as (regular) $\mathit{Ann}$-categories.  Let $\str$ be a ring-like stack with $\pi_0(\str)=B$ and $\pi_1(\str)=A$, which means that, by Theorem~\ref{thm:1}, $\str$ determines a 2-term extension of the type recalled above. The class of this extension is, locally, an element of the third Shukla cohomology $\SH^3(B/k,A)$. It is well known that it is isomorphic to the second Barr-André-Quillen cohomology $\DerD^2_k(B,A)$. By sheafification we get an object $\shDerD^2_k(B,A)$  of $\T$. We have:
\begin{proposition*}[Proposition~\ref{prop:3}]
  A ring-like stack $\str$ determines a class $[\str]$ which is a global section of $\shDerD^2_k(B,A)$.
\end{proposition*}

\subsection*{Organization of the material}
\label{sec:organization}

We define ring-like stacks in section~\ref{sec:prel-main-def}, right after having recalled the relevant facts about the monoidal structure of the 2-category of Picard stacks (after Deligne). The definition and other known, but necessary, facts about crossed bimodules are collected in section~\ref{sec:crossed-bimodules}. The fact that every ring-like stack is equivalent to one associated to a presheaf of simplicial rings determined by a crossed bimodule, i.e.\ that every ring-like stack admit a presentation by a crossed bimodule, is proved in section~\ref{sec:pres-ring-like}. There, we also leverage the presentation to obtain formulas for the products of objects in a ring-like stack; we explicitly write these formulas in terms of descent data (i.e.\ cocycles) representing the objects.

Butterflies between crossed bimodules, which we call bimodule butterflies, are defined in section~\ref{sec:bimodule-butterflies-sec}, and we prove our main result on the representation of morphisms between ring-like stacks in section~\ref{sec:butt-morph-ring}.  The equivalence between the 2-category of ring-like stacks and the bicategory of crossed bimodules is discussed in section~\ref{sec:comp-2stack-xbmod}.

Finally, we make the connection with Shukla cohomology in Section~\ref{sec:shukla}. Most of the material is well known, and our intention is merely to illustrate the connection, hence our presentation is by necessity quite sketchy.

In an appendix we reproduce an argument by T.~Pirashvili to show that the definition of Shukla cohomology of ref.~\cite{MR2270566}, which is used in sect.~\ref{sec:shukla}, agrees with the associative algebra cohomology defined by Quillen in ref.~\cite{MR0257068}.

\subsection*{Notation and terminology}
\label{sec:notation-terminology}
For the hierarchy of commutativity conditions on monoidal (or actually group-like) categories and stacks we use the terms: braided, symmetric, and Picard as opposed to braided, Picard, and strictly Picard in force in, \eg \cite{0259.14006,MR95m:18006,MR1702420}.

All complexes are cohomological, that is, the differential raises the degree. In order to simplify our notation, we use lower indices for negative degrees.  In particular, for crossed (bi)modules we denote $\del\colon M\to R$, or rather the corresponding complex, by $R_\bullet$, with $R_0=R$ and $R_1=M$.

We fix a site $\s$  and the topos $\T=\shs$ of sheaves over $\s$.  A set-theoretic notation is employed.  If $F$ is an object of $\T$, then $x\in F$ means $x\in F(U)$ for an appropriate (but not relevant) $U\in \Ob(\s)$, or equivalently $x\colon U\to F$, identifying $U$ with the (pre)sheaf it represents. The same holds for the notation $x\in \str$ when $\str$ is a (pre)stack.

For simplicial manipulations we use Duskin's ``opposite index convention'' or “missing index” convention \cite[pages 207–210]{MR1897816}, with the variant that we reverse the indexing for the 1-simplices.

Finally, regarding the question of terminology, the term $\mathit{Ann}$-category was coined and used in \cite{MR2458668,MR2065328} by analogy with the better known ``$\mathit{gr}$-category,'' used to denote a 2-group, or categorical group, i.e.\ a group-like groupoid. The axioms of an $\mathit{Ann}$-category are those of a ring structure up to isomorphism and coherence, and, quite importantly, they also require the underlying group-like groupoid to be Picard. (Again, there is a terminology shift in the recent \cite{MR3085798}, which uses \emph{regular} $\mathit{Ann}$-categories.)

On the other hand, the term ``categorical ring'' often refers to a bimonoidal structure which is also ``ring-like,'' but where the underlying categorical group is only required to be symmetric \cite[see e.g.][]{MR2369166}, or even just braided, \cite{biext2015}, as opposed to strictly Picard.

Here we do not consider these more general alternatives and restrict ourselves to the Picard case, hence we use ``ring,'' or ``ring-like''-category, or ``categorical ring,'' as a strict synonym of the term $\mathit{Ann}$-category.

\subsection*{Acknowledgements}
\label{sec:acknowledgements}

I would like to thank T. Pirashvili for answering questions related to Shukla cohomology and for providing the argument reproduced in the appendix.

\section{Ring-like stacks}
\label{sec:prel-main-def}

A \emph{ring-stack,} or stack with a ring-like structure, will be a stack $\str$ in groupoids over a site $\s$ equipped with a structure making it into a so-called categorical ring.  There are different
non-equivalent definitions of such a notion, according to whether the
underlying category or stack is Picard, or merely symmetric.  Our
current stance is to assume $\str$ to be \emph{Picard,} thereby the
resulting ring-like fiber category will be akin to the Ann-categories
of ref.~\cite{MR2458668}, as opposed to those of
ref.~\cite{MR2369166}.

\subsection{ Tensor products of Picard stacks}
\label{sec:pic-stacks-tensor}

In ref.\ \cite{0259.14006} Deligne observes that the 2-category
$\PIC(\s)$ of Picard stacks over the site $\s$ is equipped with a
monoidal structure $\otimes \colon \PIC(\s)\times \PIC(\s) \to
\PIC(\s)$.  Recall that each Picard stack $\sta$ admits a presentation
\begin{equation*}
  [A_1\overset{\del}{\lto} A_0]\sptilde\lisoto \sta
\end{equation*}
where $A_1\to A_0$ is a complex of abelian sheaves on $\s$ supported
in degrees $[-1,0]$ and the left hand side above denotes forming the
associated stack. (This is the same as taking the stacky quotient
$\bigl[ A_0/A_1\bigr]$ by the action of $A_1$ on $A_0$ via the
differential of the complex.)

The tensor structure in $\PIC(\s)$ is defined as follows. If $\sta$ and $\stb$
are Picard stacks with given presentations as above, define \cite{0259.14006}:
\begin{equation*}
  \sta\otimes \stb =
  \Bigl[ \tau_{\geq -1} \bigl( A_\bullet \overset{\derL}{\otimes}
  B_\bullet \bigr)
  \Bigr]\sptilde,
\end{equation*}
where $\tau$ denotes the soft truncation. The construction of $\sta\otimes \stb$ does indeed have the expected universal property with respect to biadditive functors from $\sta\times \stb$.  In
slightly more details, for any Picard stacks $\stp$ and $\stq$ let
$\shcatHom(\stq,\stp)$ denote the Picard stack of additive
functors. Moreover, let $\shcatHom (\sta, \stb; \stp)$ denote the
Picard stack of biadditive functors.  Then there is an equivalence (of
Picard stacks)
\begin{equation*}
  \shcatHom(\sta, \stb; \stp) \iso \shcatHom (\sta\otimes \stb, \stp),
\end{equation*}
and $\otimes\colon \sta\times \stb \to \sta\otimes\stb$ is a ``universal biadditive functor.''

\subsection{Definition of ring-like stacks}
\label{sec:def-ring-like}

With these premises, define a ring-like stack by mimicking the well-known fact that a ring is a monoid in the monoidal category of abelian groups:
\begin{definition}
  \label{def:1}
  A ring-like stack over $\s$ is a Picard stack $\str$ over $\s$
  equipped with a morphism
  \begin{equation*}
    m\colon \str\otimes \str \lto \str
  \end{equation*}
  of Picard stacks and a (unit) object $I$ which together combine into
  the structure of a (lax) monoid in $\bigl(\PIC (\s),\otimes\bigr)$.
\end{definition}
In the sequel we will usually suppress $m$ from the notation, and
write $XY$ in place of $m(X\otimes Y)$, etc.
\begin{remark}
  \label{rem:1}
  If $\str$ is a ring-like stack as above and $[\del \colon M\to R]\sptilde$ is a presentation of the underlying Picard stack, then the objects of $\str$ can be interpreted as $M$-torsors with a trivialization of their extension as $R$-torsors (cf.\ refs.\ \cite{MR92m:18019,ButterfliesI}). It is important to keep in mind that this interpretation pertains to the \emph{additive} structure of $\str$. Thus, if $E_1$ and $E_2$ are two objects of $\str$, the object $E_1+E_2$ corresponds in this interpretation to the standard torsor contraction $E_1\cprod{M} E_2$.  On the other hand, the object $E_1E_2$ is an altogether different one (see below for an explicit construction).
\end{remark}
\begin{example}
  \label{ex:1}
  Let $\sta$ be a Picad stack, and let $\str = \shcatEnd (\sta)$, the Picard stack of endomorphisms with respect to the sum of additive functors induced by that of $\sta$.  Then $\str$ has a ring-like structure with multiplication given by composition.
\end{example}

A morphism of ring-like stacks $F\colon \sts \to \str$ is defined in the obvious way, that is, as a morphism of the underlying Picard stacks compatible with the $\otimes$-monoidal structures, see~\cite{MR2369166}—modulo the difference between the symmetric and the Picard conditions for the underlying categorical groups.

\section{Crossed bimodules and their quotients}
\label{sec:crossed-bimodules}

\subsection{Crossed bimodules}
\label{sec:crossed-bimodules-1}

A way to produce ring-like stacks in the above sense is to consider
complexes equipped with some additional structure, and then take the
associated stack in the usual way. The appropriate structure is that
of a \emph{crossed bimodule,} or crossed module in algebras over
$\s$.  Let us use the notation $\ch (\s)$ for the category of
complexes of abelian sheaves on $\T$ supported in degrees
$[-1,0]$. Let us recall the definition (cf.\ refs.\
\cite{MR1600246,MR1918184,MR2270566}).
\begin{definition}
  \label{def:2}
  A crossed bimodule, or algebra crossed module, of $\ch (\s)$ is a
  complex
  \begin{equation*}
    \del\colon M\lto R
  \end{equation*}
  where $R$ is a ring, $M$ is an $R$-bimodule, and $\del$ is a
  morphism of $R$-bimodules such that
  \begin{equation}
    \label{eq:8}
    (\del m_1)m_2 = m_1(\del m_2).
  \end{equation}
  for all $m_1,m_2\in M$.
\end{definition}
It is clear that the definition works for $k$-algebras, where
$k$ is a fixed commutative ring $\T$.
\begin{remark}
  \label{rem:2}
  In more intrinsic term, the last condition in the definition—the Pfeiffer identity in algebra form—amounts to the commutativity of
  \begin{equation*}
    \xymatrix@C+1pc{%
      M \times M \ar[r]^{(\id_M,\del)} \ar[d]_{(\del,\id_M)} & M\times R \ar[d] \\
      R \times M \ar[r] & M
    }
  \end{equation*}
  In fact the resulting morphism is $R$-bilinear, hence it induces a product map
  \begin{align*}
    M\otimes_R M & \lto M\\
    m\otimes m' & \longmapsto m\, \del m'
  \end{align*}
  making $M$ into a non-unital ring (or $k$-algebra), with $\del$ becoming a homomorphism of non-unital rings.  We will denote by $\langle\cdot,\cdot\rangle$ this map.
\end{remark}

The primary example of a crossed bimodule is that of a (bilateral) ideal $I$ in a ring $R$. Secondly, for any ring $R$ of $\T$ and any $R$-bimodule $M$,
\begin{equation*}
  M\overset{0}{\lto} R
\end{equation*}
is evidently a crossed bimodule.

A good supply of crossed bimodules is provided by differential rings or differential graded $k$-algebras or simplicial rings (or $k$-algebras), depending on the framework we choose, as follows.
\begin{example}[See~\cite{BFM-CFT}]
  \label{ex:2}
  Let $R^\bullet$ be a $k$-DGA supported in negative degrees, that is,
  $R^i=0$ for $i>0$. Then the soft truncation
  \begin{equation*}
    \tau_{\geq -1} R^\bullet \colon (R^{-1}/\im \del) \lto R^0
  \end{equation*}
  is a crossed bimodule.
\end{example}
\begin{example}
  Let $\smp{R}_\bullet$ be a simplicial ring.  Let $MR^\bullet$ be its Moore complex (denoted cohomologically): in each degree $MR^{-n} = \cap_{i=0}^{n-1} \ker d_i$, with $d=d_n$ restricted to $MR^{-n}$.  It is easily verified that the soft truncation
  \begin{equation*}
    \tau_{\geq -1} MR^{\bullet} \colon MR^{-1}/\im d \lto MR^0 = R_0
  \end{equation*}
  is a crossed bimodule.  For this, let $R_0$ act on $MR^{-1}=\ker d_0$ by
  \begin{equation*}
    r_0\cdot m\cdot r_1 \eqdef s_0(r_0)\, m\, s_0(r_1),
  \end{equation*}
  where $r_0,r_1\in R_0$ and $m\in \ker d_0$. In addition, if $m,m'\in R_1$, then the simplicial identities imply that the combination $d_1(m)\cdot m' - m\cdot d_1(m') = s_0d_1(m)m'-ms_0d_1(m')$ belongs to $\im d_2\colon R_2\to R_1$, since
  \begin{equation*}
    s_0d_1(m)m'-ms_0d_1(m') = d_2\bigl( s_0(m)s_1(m') - s_1(m)s_0(m')\bigr).
  \end{equation*}
  Furthermore, if $m,m'\in \ker d_0$, the combination within the parentheses on the right hand side above belongs to $MR^{-2}$.  Thus in the soft truncation the algebraic Pfeiffer identity~\eqref{eq:8} is satisfied.  
\end{example}

\begin{paragr}
The crossed bimodule $\del \colon M\to R$ determines a groupoid
\begin{equation*}
  \str_0 \colon \xymatrix@1{ R \oplus M \ar@<0.6ex>[r] \ar@<-0.3ex>[r] & R },
\end{equation*}
objectwise over $\s$, which is a presheaf of strict categorical rings: the additive structure is standard, and the multiplicative one is given, at the level of objects, by the ring structure of $R$, and at the level of morphisms by
\begin{equation}
  \label{eq:11}
  (r_0,m_0)(r_1,m_1) = (r_0r_1, r_0m_1 + m_0r_1 + m_0\del(m_1)),
\end{equation}
for all $r_0,r_1\in R$ and $m_0,m_1\in M$.  The verification of the axioms is straightforward.  The nerve of $\str_0$ (again, objectwise), is a simplicial presheaf $\N_\bullet\str_0$ where, for each $n\geq 0$,
\begin{equation*}
  \N_n \str_0 = R\oplus \underbrace{M \oplus \dots \oplus M}_{n-\text{times}}.
\end{equation*}
It is easy to see that $\N_\bullet\str$ is a simplicial ring.  For this, analogously to ref.\ \cite{MR92m:18019}, inductively define $u_n\colon \N_n\str_0 \to \N_0\str_0 = R$ by
\begin{equation*}
  u_0 = \id_R\,,\qquad u_n(y,m) = u_{n-1}(y) + \del m,
\end{equation*}
where we write an object of $\N_n\str_0 \iso \N_{n-1}\str_0 \oplus M$
as $(y,m)$, with $y\in \N_{n-i}\str_0$ and $m\in M$.  Then the ring
structure is obtained by inductively generalizing~\eqref{eq:11},
namely with the same conventions:
\begin{equation*}
  (y_0,m_0)(y_1,m_1) = (y_0y_1, u_{n-1}(y_0)m_1 + m_0u_{n-1}(y_1) + m_0\del(m_1)).
\end{equation*}
\end{paragr}

In particular, crossed bimodules are seen in this way to be equivalent to simplicial rings whose Moore complexes are supported in degrees $[-1, 0]$.

\begin{paragr}
  \label{prgr:1}
If $\del\colon M\to R$ is a crossed bimodule, one considers $A=\pi_1(R_\bullet)=\ker \del$ and $B=\pi_0(R_\bullet) =\coker \del$. It is well known and easy to see that $B$ is a ring (or $k$-algebra) and $A$ a $B$-bimodule.  One refers to the complete exact sequence
\begin{equation*}
  \xymatrix@1{%
    0 \ar[r] & A\ar[r] & M \ar[r]^\del & R \ar[r] & B \ar[r] & 0
  }
\end{equation*}
as a \emph{crossed extension} of $B$ by $A$.  Of course $A$ and $B$ are the homotopy objects of the simplicial ring determined by the crossed bimodule, in other words the homology objects of the associated Moore complex.
\end{paragr}
\begin{example}[Variant of~\ref{ex:2}]
  \label{ex:3}
  Let $\del\colon M\to R$ be a crossed bimodule as in~\ref{prgr:1}. The chain complex
  \begin{equation*}
    \xymatrix@1{%
      \cdots \ar[r] & 0 \ar[r] & A\ar[r]^\imath & M \ar[r]^\del & R}
  \end{equation*}
  is equipped with a $k$-DGA structure: define $x_ix_j$ via the bimodule structure if either $i=0$ or $j=0$, and $x_ix_j=0$ otherwise. It is quasi-isomorphic to $B$, with its structure of $k$-DGA concentrated in degree zero.  See \cite{MR2270566} for more details on the relations between $k$-DGAs and crossed bimodules.
\end{example}

\subsection{Strict morphisms}
\label{sec:strict-morphisms}

The notion of  morphism between crossed bimodules has a straightforward definition.
\begin{definition}
  \label{def:3}
  Let $\del \colon M\to R$ and $\del\colon N\to S$ be two crossed bimodules.  A morphism of crossed bimodules between them is a morphism of complexes, \ie a pair $(\alpha,\beta)$ such that in the commutative diagram
  \begin{equation*}
    \xymatrix{%
      N \ar[r]^\beta \ar[d]_\del & M \ar[d]^\del \\ S \ar[r]_\alpha & R
    }
  \end{equation*}
  $\alpha$ is a ring homomorphism and $\beta$ is $\alpha$-equivariant, that is, we have $\beta(s_0\,n\,s_1)=\alpha(s_0)\beta(n)\alpha(s_1)$ for all $s_0,s_1\in S$, $n\in N$.
\end{definition}
By a standard procedure, a morphism of crossed bimodules will induce a functor between the corresponding groupoids.  It is straightforward to verify that it is a morphism of ring-like structures. For, let $\sts_0$ and $\str_0$ be the groupoids corresponding to the complexes $N\to S$ and $M\to R$, respectively.  It is standard that $\alpha$ and $(\alpha,\beta)\colon S\oplus N\to R\oplus M$ combine to give an additive functor $\sts_0\to \str_0$.  In addition, we have, for all $(s,n)$ and $(s',n')\in S\oplus N$,
\begin{align*}
  (\alpha (s),\beta (n))\otimes (\alpha (s'),\beta (n')) &=
  ( \alpha (s)\alpha(s'), \alpha(s)\beta(n')+\beta(n)\alpha(s')
  + \beta(n)\del\beta(n') ) \\
  & = (\alpha (ss'), \beta(sn' + ns') + \beta(n)\alpha(\del n')) \\
  & = (\alpha (ss'), \beta(sn' + ns') + \beta(n\,\del n')),
\end{align*}
and the latter is just the image of $(s,n)\otimes (s',n')$.  Thus $(\alpha,\beta)$ gives a morphism of strict categorical rings.

We recall the notion of homotopy. Let $(\alpha,\beta)$ and $(\alpha',\beta')$ be two morphisms between $\del \colon N\to S$ and $\del \colon M\to R$, as in Definition~\ref{def:3} above.
\begin{definition}
  \label{def:8}
  A \emph{homotopy} $h\colon (\alpha',\beta') \Rightarrow (\alpha,\beta)$ is a $k$-linear map $h\colon S\to M$ such that:
  \begin{subequations}
    \label{eq:2}
    \begin{align}
      \label{eq:3}
      \alpha' -\alpha &= -\del \circ h, \\
      \label{eq:4}
      \beta' - \beta &= -h \circ \del, \\
      \intertext{and, for all $s,s'\in S$,}
      \label{eq:5}
      h(ss) &= \alpha(s)h(s') + h(s)\alpha(s') - \del h(s) \, h(s').
    \end{align}
  \end{subequations}
\end{definition}
\begin{remark}
  \label{rem:5}
  The first two conditions amount to the standard definition of chain homotopy (for complexes supported in degrees $[-1,0]$).  The third can be given the following interpretation.  Consider the complex of Hochschild cochains for $S$ with values in the $S$-bimodule $M$ (cf.\ \cite{MR1600246}, the bimodule structure is via the homomorphism $\alpha$).  Let us denote the Hochschild coboundary by $\delta$. Then~\eqref{eq:5} can be recast as
  \begin{equation*}
    \delta h = \langle h\,,\, h\rangle,
  \end{equation*}
where the right-hand side is the product introduced in Remark~\ref{rem:2}.
\end{remark}
A homotopy $h$ determines a morphism of functors $(\alpha',\beta')\Rightarrow (\alpha,\beta)$ between the categorical rings $\sts_0$ and $\str_0$. Also, it is easily verified that $\tilde h = -h$ is a homotopy from $(\alpha,\beta)$ to $(\alpha',\beta')$, thus morphisms of crossed bimodules and homotopies between them form a groupoid, denoted $\catHom(S_\bullet,R_\bullet)$.

\subsection{Associated ring-like stacks}
\label{sec:associated-stacks}

Let $\del \colon M\to R$ be a crossed bimodule, and let $\str_0$ the corresponding groupoid
\begin{equation*}
  \str_0 \colon \xymatrix@1{ R \oplus M \ar@<0.6ex>[r] \ar@<-0.3ex>[r] & R },
\end{equation*}
as above.  We have observed that it is a presheaf of categorical rings on $\s$, with corresponding strict additive bifunctor $m_0\colon \str_0\times \str_0\to \str_0$.  Let $\str=\bigl[ M\to R\bigr]\sptilde$ be the \emph{associated Picard stack,} and $j\colon \str_0 \to \str$ the corresponding local equivalence. By the usual universal property argument, we have the equivalence
\begin{equation*}
  \shcatHom(\str_0 \times \str_0, \str) \iso \shcatHom(\str\times \str, \str),
\end{equation*}
cf.\ \cite[\S 1.4.10]{0259.14006}, hence the composite morphism $j\circ m_0\colon \str_0 \times \str_0 \lto \str$ yields
\begin{equation*}
  m\colon \str \times \str \lto \str.
\end{equation*}
Thus $\str$ is a ring-like stack.

As we have observed, a  morphism of crossed bimodules $(\alpha,\beta)\colon S_\bullet \to R_\bullet$ induces one between groupoids $F_0\colon \sts_0\to \str_0$ (cf.\ sect.~\ref{sec:strict-morphisms}).  If we compose the latter with $\jmath\colon \str_0\to \str$, again by standard arguments, there results a morphism
\begin{equation*}
  F\colon \sts \lto \str
\end{equation*}
between the associated Picard stacks, which is easily seen to be a morphism of ring-like stacks.

We regard these morphisms as strict in the following sense.
\begin{definition}
  \label{def:4}
  Let $\sts\iso [N\to S]\sptilde$ and $\str \iso [M\to R]\sptilde$ be the associated stacks. A morphism $F\colon \sts\to \str$ of ring-like stacks is \emph{strict} if it arises from a crossed bimodule morphism between the presentations.
\end{definition}
An equivalent way of stating the notion of strict morphism would be to say that $F\colon \sts\to \str$ is strict whenever it  arises from a morphism of the underlying prestacks $\sts_0$ and $\str_0$.  Due to Theorem~\ref{thm:1} below, the notion of strict morphism makes sense for all ring-like stacks.

\section{Ring-like stacks and their presentations}
\label{sec:pres-ring-like}

Every Picard stack of $\T$ has a presentation by a complex of abelian sheaves supported in degrees $[-1,0]$. If $\str$ is a ring-like stack, we prove the presentation is a crossed bimodule. We use them to discuss the forms of the descent data (i.e.\ the cocycles) and the monoidal structures.
Later, in section~\ref{sec:comp-2stack-xbmod}, we discuss the significance from the point of view of the 2-category of Picard stacks. 

\subsection{Presentations of ring-like stacks}
\label{sec:pres-ring-st}

\begin{theorem}
  \label{thm:1}
  Let $\str$ be a ring-like stack. Then $\str$ admits a presentation by a crossed bimodule $\del \colon M\to R$.
\end{theorem}
\begin{proof}[Proof (Sketch)]
  Take a presentation $\str \iso [B\to A]\sptilde$ of the underlying Picard stack of $\str$. The complex $B\to A \equiv A_1\to A_0$ is just a complex of abelian groups of $\T$.

  Consider the tensor algebra $T(A)$ over $A$, where $T(\cdot)$ is taken over $\ZZ$.  We claim the projection $\pi\colon A\to \str$ factors through $T(A)$.  To see this, define $\varpi\colon T(A)\to \str$ by
  \begin{equation*}
    \varpi (a_1\otimes \dots\otimes a_n) =
    (\dots (\pi(a_1)\pi(a_2)) \dots \pi(a_n)\dots ),
  \end{equation*}
  using the left bracketing for the expression on the right. We want this to be unital, namely for $n=0$ we send $1$ to the $I_\str$, the multiplicative unit object of $\str$.  One can view the $a_i$, $i=1,\dots,n$ as parametrizing a collection of objects of $\str$ via $\pi$.  Thus, by \cite{MR0296132,MR0335598}, $\varpi$ is well defined. It is also essentially surjective, since $\pi$ is. Now, define $M$ as the homotopy kernel of $\varpi\colon T(A)\to \str$. An element of $M$ is a pair $(b_1\otimes\dots\otimes b_n,\lambda)$ where $\lambda\colon \varpi(b_1\otimes\dots \otimes b_n)\isoto 0_\str$. Forgetting $\lambda$ gives the differential $\del\colon M\to T(A)$. As it is easily seen, $M$ is a $T(A)$-bimodule, and computations similar to those in \cite[5.3.8]{ButterfliesI} show that the Pfeiffer identity holds. Thus the complex $\del \colon M\to R$, with $R=T(A)$, is a crossed bimodule. We also have a commutative diagram
  \begin{equation*}
    \xymatrix{%
      B \ar[r]^\del \ar[d]_{\eta_1} & A \ar[r]^\pi \ar[d]^{\eta_0} & \str \ar@{=}[d] \\
      M \ar[r]_\del & T(A) \ar[r]_\varpi & \str
    }
  \end{equation*}
  where $\eta_0$, $\eta_1$ are monomorphisms. It is easily seen that $\del\colon B\to A$ and $\del\colon M\to T(A)$ have the same kernel and cokernel, which then coincide with $\pi_1(\str)$ and $\pi_0(\str)$.
\end{proof}
In the following we will always use a presentation $\del\colon M\to R$ by a crossed bimodule.

\subsection{Objects and  products in a ring-like stack}
\label{sec:objects-products}

The standard geometric interpretation of $\str = \bigl[ M\to R \bigr]\sptilde$ is obtained by observing that it is equivalent (as a Picard stack) to $\tors(M,R)$, the stack of $M$-torsors $E$ equipped with a trivialization $s\colon E\cprod{M}R\isoto R$. If $E$ and $E'$ are two $M$-torsors, it is standard that their sum $E+E'$ is given by the $M$-torsor $E\cprod{M}E'$ equipped with the trivialization $s+s'$. The projection  morphism $\pi_\str\colon R\to \tors(M,R)$ assigns to $r\in R$ the trivial torsor $M$ equipped with the $M$-equivariant map that sends $0$ to $r$. (Thus $m\in M$ is sent to $r+\del m$.)  In particular, $(M,0)=\pi_\str(0)$ will be identified with the zero object $0_\str$ (the unit of the sum operation).

Less standard is the product $EE'=m(E,E')$ induced by the second monoidal structure of the ring-like stack $\str$.  This structure can be described as follows. First, a local description. To local data (\ie sections) $e\in E$ and $e'\in E'$ we assign the trivial $M$ torsor, which we can think as being generated by the symbol $\lbrace e,e'\rbrace$.  Recall that $E$ and $E'$ have trivializations of their push-outs as $R$-torsors via $\del$, by way of $M$-equivariant maps $s$ and $s'$, respectively.  The trivial $M$-torsor associated to $(e,e')$ is equipped with the map denoted $ss'$ sending the generator $\lbrace e,e'\rbrace$ to $s(e)s'(e')\in R$.  Replacing the pair $(e,e')$ with $(e+m,e'+m')$, results in the isomorphism of trivial torsors such that:
\begin{equation*}
  \lbrace e,e'\rbrace \lmto \lbrace e+m,e'+m'\rbrace - \bigl( s(e)m' + ms'(e') + m\del m'\bigr).
\end{equation*}
The sought-after $M$-torsor $EE'$ will be obtained by gluing the above trivial torsors by way of this isomorphism. (It is clear that it satisfies the appropriate cocycle condition.)

An alternative more global description simply is obtained by observing that the above construction presents $EE'$ as the quotient of $E\times E'\times M$ by the action of $M\times M$ given by:
\begin{equation*}
  (e,e',m_0) \lmto (e+m,e'+m', m_0 + s(e)m' + ms'(e') + m\del m'), \quad m,m'\in M.
\end{equation*}
The correspondence between the two pictures is that $\lbrace e,e'\rbrace$ is the class of $(e,e',0)$, and that in the resulting $M$-torsor we have:
\begin{equation*}
  \lbrace e+m,e'+m'\rbrace = \lbrace e,e'\rbrace+ \bigl( s(e)m' + ms'(e') + m\del m'\bigr).
\end{equation*}
Note that the map $ss$ defined above is compatible with this relation, hence it is well defined as an $M$-equivariant map $ss\colon EE'\to R$.

The unit object $I_\str$ for the just defined multiplicative structure can be identified with $(M,1)=\pi_\str(1)$. Indeed, if $E$ is any $(M,R)$-torsor, we have the standard structure isomorphisms:
\begin{align*}
  \lambda_E \colon I_\str E & \lto E & \rho_E \colon E\; I_\str & \lto E \\
  \lbrace 0,e \rbrace & \lmto e, & \lbrace e,0\rbrace & \lmto e.
\end{align*}
It is easily checked that they are well defined and functorial.

\subsection{Cocycles}
\label{sec:cocycles}

Objects of $\str$ can be described in terms of descent data.  Given a presentation, descent data become just cocycle representations for such torsors as described above.  Using these data, the ring structure on $\str$ is very concretely described by localized versions of the formulas for $\str_0$, as follows.

Let $V_\bullet\to U$ be a hypercover of $\T$. An object $E$ over $U$ will be represented by a triple $(V_\bullet,r,m)$, where $r\in R(V_0)$, and $m\in M(V_1)$ such that
\begin{align*}
  d_0^*r - d_1^*r &= \del m \\
  d_2^*m + d_0^*m &= d_1^*m.
\end{align*}
If now $E, E'$ are two objects of $\str_U$, and $(V_\bullet,r,m)$ and $(V_\bullet,r',m')$ the corresponding descent data, where the hypercover $V_\bullet\to U$ is assumed for simplicity to be same for both objects, the object $E+E'$ is represented by $(V_\bullet,r_0+r_1, m_0+m_1)$, whereas the multiplication $EE'=m(E,E')$ is represented by the triple
\begin{equation*}
  \bigl(V_\bullet, rr', (d_1^*r)m' + m(d_0^*r') + m\del m'\bigr) =
  \bigl(V_\bullet, rr', (d_1^*r)m' + m(d_1^*r') \bigr).
\end{equation*}
These formulas are most transparent in the Čech formalism. Assuming $\T = \mathrm{Sh}(\s)$, and $\s$ to have (finite) limits, if $(U_i)_{i\in I}$ is a cover of $U\in \s$, we write the data for $E$ as a collection $(r_i, m_{ij})$, where $r_i\in R(U_i)$ and $m_{ij}\in M(U_i\times_{U}U_j)$, such that
\begin{align*}
  r_j -r_i &= \del m_{ij},\\ m_{ij} + m_{jk}  &= m_{ik},
\end{align*}
and similarly for $E'$.  Therefore $E+E'$ is represented by
\begin{equation*}
  (r_i+r'_i, m_{ij} + m'_{ij},
\end{equation*}
whereas $EE'$ by
\begin{equation*}
  (r_ir'_i, r_im'_{ij} + m_{ij}r'_j).
\end{equation*}

The cocycle, that is, the triple $(V_\bullet,r,m)$ corresponds to a simplicial map
\begin{equation*}
  \xi \colon V_\bullet \lto \N_\bullet\str_0,
\end{equation*}
see \cite{MR2597739} and \cite[\S 3.3.1-3.4.4]{ButterfliesI}. The simplicial ring structure of $\N_\bullet\str_0$ gives pointwise sum and product operations for cocycles. Hence $E+E'$ and $EE'$ give rise and are determined by the simplicial maps $\xi+\xi'$ and $\xi\xi'$, defined by
\begin{equation*}
  (\xi + \xi')_n \eqdef \xi_n+\xi'_n \quad \text{and} \quad (\xi\xi')_n\eqdef \xi_n\xi'_n.
\end{equation*}
By explicitly writing down the simplicial maps (see \loccit or, \eg, \cite{MR1206474}) we arrive at the formulas for the addition and multiplication of cocycles given above.

\section{Bimodule butterflies}
\label{sec:bimodule-butterflies-sec}

\emph{Butterflies} (\cite{Noohi:notes,ButterfliesI}) are certain kind of diagrams computing morphisms between length 2-complexes in the homotopy category.  We specialize the concept to the present situation.

\subsection{Bimodule butterflies}
\label{sec:bimodule-butterflies}

Let $S_\bullet \colon \xymatrix@1{N\ar[r]^\del& S}$ and $R_\bullet \colon \xymatrix@1{ M\ar[r]^\del & R}$ be two crossed bimodules of $\T$.
\begin{definition}
  \label{def:5}
   A \emph{crossed bimodule butterfly,} or simply a \emph{butterfly,} for short, from $S_\bullet$ to $R_\bullet$ is a diagram 
  \begin{equation}
    \label{eq:1}
    \vcenter{\xymatrix@-1pc{%
        N \ar[dd]_\del \ar[dr]^\kappa && M \ar[dl]_\imath \ar[dd]^\del \\
        & E \ar[dl]_\pi \ar[dr]^\jmath \\
        S && R
      }}
  \end{equation}
  where:
  \begin{enumerate}
  \item $E$ is a ring (or $k$-algebra);
  \item The NE-SW diagonal $\xymatrix@1@C-1pc{M \ar[r] &E \ar[r] & S}$ is an extension, namely it is an exact sequence of the underlying modules, and $\pi \colon E\to S$ is a ring (or $k$-algebra) homomorphism;
  \item The NW-SE diagonal $\xymatrix@1@C-1pc{N \ar[r] & E\ar[r] & R}$ is a complex of abelian groups (or $k$-modules), namely $\jmath\circ \kappa=0$; $\jmath\colon E\to E$ is a ring ($k$-algebra) homomorphism;
  \item For all $m\in M$, $n\in N$, and $e\in E$, the following compatibility conditions hold:
    \begin{subequations}
      \label{eq:7}
      \begin{align}
        \imath (m\, \jmath (e)) &= \imath (m)\, e \\
        \imath (\jmath (e)\, m) &= e\, \imath (m) \\
        \kappa (n\, \pi (e)) &= \kappa (n)\, e \\
        \kappa (\pi (e)\, n) &= e\, \kappa (n).
      \end{align}
    \end{subequations}
  \end{enumerate}
\end{definition}
There are some elementary consequences of the definition.
\begin{lemma}
  \label{lem:1}
  In the butterfly defined above:
  \begin{enumerate}
  \item $M$ is a bilateral ideal in $E$.
  \item The images of $N$ and $M$ in $E$ multiply to zero: $\kappa(N)\imath(M)=0$ in $E$.
  \end{enumerate}
\end{lemma}
\begin{proof}
  The first is obvious (it is just a restatement of the second condition in the definition). The second easily follows from~\eqref{eq:7}.
\end{proof}
\begin{remark}
  \label{rem:4}
  \begin{enumerate}
  \item   The NW-SW diagonal is not necessarily an abelian extension, namely $M^2\neq 0$ in general, as an ideal in $E$. Indeed, for all $m,m'\in M$ we have
    \begin{equation*}
      \imath(m)\imath(m') = \imath(m\,\jmath\imath (m')) = \imath( m\, \del m'),
    \end{equation*}
    and $m\,\del m'$ is in general nonzero.
  \item The multiplication on $M$ induced by $E$ is the same as that induced by the crossed bimodule structure (cf.\ Remark~\ref{rem:2}).
  \end{enumerate}
\end{remark}

A shorthand notation for a butterfly \emph{from} $S_\bullet$ \emph{to} $R_\bullet$ with centerpiece $E$ will be $\bigl[ S_\bullet, E, R_\bullet\bigr]$.
\begin{definition}
  \label{def:6}
  A \emph{morphism} $\alpha \colon \bigl[ S_\bullet, E, R_\bullet\bigr] \to \bigl[ S_\bullet, E', R_\bullet\bigr]$ is a ring (or $k$-algebra) isomorphism $\alpha\colon E\isoto E'$ compatible with the structural maps of both butterflies in the sense that the following diagram commutes
  \begin{equation*}
    \xymatrix@-.5pc{%
      N \ar[dd]_\del \ar[ddrr]^<<<<\kappa \ar[rr]^{\kappa'}& & E' \ar '[dl] [ddll] \ar '[dr] [ddrr] & & M \ar[ll]_{\imath'} \ar[ddll]_<<<<\imath \ar[dd]^\del \\
      & & & & \\
      S & & E \ar[ll]_\pi \ar[rr]^\jmath  \ar[uu]_\alpha &  & R
    }
  \end{equation*}
\end{definition}
With the notion of morphism just introduced, butterflies from $S_\bullet$ to $R_\bullet$ clearly form a groupoid, denoted $\B (S_\bullet,R_\bullet)$.  Analogously to~\cite[\S 5 and \S 8]{ButterfliesI}, we can consider a local version with respect to $\s$, namely form the fibered category $\stb (S_\bullet,R_\bullet)$ from $U\mapsto \B (S_\bullet\vert_U,R_\bullet\vert_U)$, where $U\in \Ob (\s)$.  These groupoids are subgroupoids of the corresponding ones constructed by forgetting the multiplicative structures and considering just the underlying abelian sheaves (or $k$-modules). Denote them by $\B_k(S_\bullet,R_\bullet)$ and $\B_k (S_\bullet\vert_U,R_\bullet\vert_U)$, respectively. By \loccit, the latter are the fibers of a stack $\stb_k (S_\bullet, R_\bullet)$. This implies that $\stb (S_\bullet, R_\bullet)$ forms a stack as well. Note, however, that it will not be closed with respect to the symmetric structure given by ``addition'' of butterflies (cf.~\cite[\S 8]{ButterfliesI}).

\subsection{Fractions}
\label{sec:fract-determ-butt}

The diagram~\eqref{eq:1} can be completed to
\begin{equation*}
  \vcenter{\xymatrix@-1pc{%
      & N\oplus_S E \ar[dl]_{\tilde\pi} \ar[dr]^{\tilde\jmath} \ar[dd]^\del \\
      N \ar[dd]_\del \ar[dr]^\kappa && M \ar[dl]_\imath \ar[dd]^\del \\
      & E \ar[dl]_\pi \ar[dr]^\jmath \\
      S && R
    }}
\end{equation*}
where the left wing is a pull-back. With set-theoretic notation, $N\oplus_S E = \lbrace (n,e) \in N\oplus E \; \vert\; \del n = \pi e \rbrace$.  As in the abelian case, $\kappa\colon N\to E$ gives a splitting of $\tilde\pi$, so that we have an isomorphism
\begin{equation*}
  (\id_N,\id_E-\kappa)\colon  N\oplus_S E \lisoto N\oplus M,
\end{equation*}
with inverse $(\id_N,\imath + \kappa)$.  In addition, the complex $E_\bullet\colon N\oplus_S E\to E$ is a crossed bimodule: first, $N\oplus_S E$ is an $E$-bimodule with the operations (written set-theoretically as):
\begin{equation*}
  e_0\cdot (n,e) = (\pi(e_0)n, e_0e) \quad\text{and} \quad
  (n,e)\cdot e_1 = (n\pi(e_1), ee_1),
\end{equation*}
for all $e,e_0,e_1\in E$ and $n\in N$.  An elementary verification shows that the Pfeiffer identity (cf.\ Remark~\ref{rem:2})
\begin{equation*}
  \del (n_0,e_0) \, (n_1,e_1) = (n_0,e_0)\,\del(n_1,e_1)
\end{equation*}
holds.
\begin{lemma}
  \label{lem:2}
  Each wing of the above diagram determines a morphism of crossed bimodules, the left one being a quasi-isomorphism.
\end{lemma}
\begin{proof}
  The first statement is an elementary verification and it is left to the reader. The second follows from considering the pullback of extensions
  \begin{equation*}
    \xymatrix{%
    0\ar[r] & M \ar[r]^{\tilde\imath} \ar@{=}[d] & N\oplus_S E \ar[r]^{\tilde\pi} \ar[d]^{\del_E} & N \ar[r] \ar[d]^{\del_S} & 0 \\
    0\ar[r] & M \ar[r]^{\imath} & E \ar[r]^{\pi} & S \ar[r] & 0
    }
  \end{equation*}
  along $\del_S$.  An elementary application of the snake lemma yields $\pi_i(S_\bullet)\iso \pi_i(E_\bullet)$, for $i=0,1$.
\end{proof}

\subsection{Split butterflies}
\label{sec:strict-morphism}

A morphism $(\alpha,\beta)$ of crossed bimodules determines a butterfly in which the NE-SW diagonal is a trivial extension, namely
\begin{equation*}
  E = S \oplus M,
\end{equation*}
where $M$ is considered as an $S$-bimodule via $\alpha\colon S\to R$. The ring structure on $E$ is given by
\begin{equation*}
  (s,m)\,(s',m') = (ss', \alpha(s)m'+m\alpha(s') + m\,\del(m')),
\end{equation*}
and the four maps in the butterfly diagram are given by:
\begin{align*}
  \imath \colon M &\lto S\oplus M & \pi\colon S\oplus M &\lto S \\
  m &\longmapsto (0,m)& (s,m) &\longmapsto s \\ \\
  \kappa \colon N & \lto S\oplus M & \jmath \colon S\oplus M &\lto R \\
  n &\longmapsto (\del n, -\beta (n)) & (s,m) &\longmapsto \alpha(s)+\del m.
\end{align*}
The map $\sigma=(\id_S,0) \colon S\to S\oplus M$ is evidently a splitting of the exact diagonal.  More generally we have:
\begin{definition}
  \label{def:7}
  A butterfly~\eqref{eq:1} is \emph{strongly split(table)} if its NE-SW diagonal is equipped with an algebra extension splitting homomorphism $\sigma: S\to E$. Equivalently, it is isomorphic in the sense of Definition~\ref{def:6} to one arising from a morphism of crossed bimodules.
\end{definition}
Thus, a strongly split butterfly in effect corresponds to a morphism of crossed bimodules. Note that such an object is in fact a \emph{pair} $(E,\sigma)$, where $E$ is an object of $\B(S_\bullet, R_\bullet)$ and $\sigma$ is an algebra splitting.  It is easily seen that a homotopy $h\colon (\alpha',\beta')\Rightarrow (\alpha,\beta)$ of morphisms of crossed bimodules determines a morphism $\psi\colon (E,\sigma)\to (E',\sigma')$ of split butterflies.  Explicitly, if both $E$ and $E'$ are identified with $S\oplus M$, then the required homomorphism has the form.
\begin{equation*}
  \psi = (\id_S,\id_M+h) \colon S\oplus M\lto S\oplus M.
\end{equation*}
Conversely, an isomorphism $\psi\colon S\oplus M\to S\oplus M$ which fits into a morphism of (split) butterflies, necessarily has the above form, with $h\colon S\to M$ satisfying~\eqref{eq:2}.

Let us denote by $\B_{\mathrm{str}} (S_\bullet,R_\bullet)$ the resulting groupoid. By the foregoing, it is equivalent to the previously introduced  groupoid $\catHom(S_\bullet,R_\bullet)$.  There is an obvious functor $\B_{\mathrm{str}}(S_\bullet, R_\bullet) \to \B (S_\bullet,R_\bullet)$, and hence $\catHom(S_\bullet,R_\bullet) \to \B (S_\bullet,R_\bullet)$.   A better characterization will be given below.

\subsection{Butterflies and extensions}
\label{sec:butt-extens}

Let us denote by $\catExtAlg(S,M)$ the category of algebra extensions of $S$ by $M$, whose objects are algebra extensions as above, and whose morphisms are commutative diagrams
\begin{equation*}
  \xymatrix{%
    0\ar[r] & M \ar[r]^\imath \ar@{=}[d] & E \ar[r]^\pi \ar[d]^\alpha & S \ar[r] \ar@{=}[d] & 0 \\
    0\ar[r] & M \ar[r]^{\imath'} & E' \ar[r]^{\pi'} & S \ar[r] & 0
  }
\end{equation*}
The extensions we consider are not assumed to be abelian, nor are they assumed to be $k$-split.  Analogously to~\cite[\S 8]{ButterfliesI}, there is an obvious forgetful functor
\begin{equation*}
  p\colon \B(S_\bullet, R_\bullet) \lto \catExtAlg(S,M)
\end{equation*}
which is a fibration (cf.\cite{MR93k:18019}). For, if $\bigl[ S_\bullet, E', R_\bullet \bigr]$ is such that its NE-SW diagonal is isomorphic to the extension $0\to M\to E \to S \to 0$ with $\alpha \colon E\to E'$, then $\bigl[ S_\bullet, E , R_\bullet\bigr]$ is an isomorphic butterfly with structure maps $\jmath=\jmath' \circ \alpha$ and $\kappa = \kappa' \circ \alpha^{-1}$. Evidently $\alpha$ gives the corresponding morphism of butterflies.  Essential surjectivity also holds, since, rather trivially, in the extension $0\to M\to E\to S\to 0$ the morphism $M\to E$ is a crossed bimodule, and so is $0\to S$, therefore we can choose
\begin{equation*}
    \vcenter{\xymatrix@-1pc@C+1pc{%
       0 \ar[dd] \ar[dr] && M \ar[dl]_\imath \ar[dd]^\imath \\
        & E \ar[dl]_\pi \ar[dr]^\id \\
        S && E
      }} 
\end{equation*}
Note that $\bigl[ M\to E\bigr]\sptilde \iso S$, where $S$ is considered as a discrete stack, since the groupoid determined by $M\to E$ is an equivalence relation.  The groupoid $\B_{\mathrm{str}}(S_\bullet,R_\bullet)$ is the homotopy kernel of the morphism $p$ above. In fact it is easy to see the whole sequence
\begin{equation*}
  \xymatrix@1{%
    \B_{\mathrm{str}}(S_\bullet,R_\bullet) \ar[r] & \B(S_\bullet, R_\bullet) \ar[r]^<<<p & \catExtAlg(S,M)
  }
\end{equation*}
is exact.  Since the homotopy kernel is determined up to equivalence, we can rewrite this sequence as
\begin{equation*}
  \xymatrix@1{%
    \catHom (S_\bullet,R_\bullet) \ar[r] & \B(S_\bullet, R_\bullet) \ar[r]^<<<p & \catExtAlg(S,M)
  }.
\end{equation*}
By forgetting the algebra structure we get (with corresponding meanings of the symbols)
\begin{equation*}
  \xymatrix@1{%
    \catHom_{k}(S_\bullet,R_\bullet) \ar[r] & \B_{k}(S_\bullet, R_\bullet) \ar[r]^<<<p & \catExt_k(S,M)
  }
\end{equation*}
which, as observed in \cite{ButterfliesI}, is also an extension. The first object on the left is identified with the groupoid of split butterflies, \ie strict morphisms of complexes of abelian sheaves.

We can also consider the (homotopy) kernel of the forgetful functor
\begin{equation*}
  \catExtAlg(S,M) \lto \catExt_k(S,M),
\end{equation*}
which we denote by $\catExtAlg_0(S,M)$.  It consists of those algebra extensions which possess a $k$-linear splitting.  The pullback groupoid
\begin{equation*}
  \xymatrix{%
    \B_0(S_\bullet,R_\bullet) \eqdef \B (S_\bullet,R_\bullet) \times_{\catExtAlg(S_\bullet,R_\bullet)} \catExtAlg_0(S_\bullet,R_\bullet)
  }
\end{equation*}
then consists of those butterflies whose NE-SW diagonal admits a $k$-linear splitting.

The above constructions can be sheafified (or actually stackified) over $\s$.  Putting all together, we can form the diagram of stacks over $\s$:
\begin{equation*}
  \xymatrix{%
    & \stb_0(S_\bullet, R_\bullet) \ar[d] \ar[r]^<<<p \ar@{}[dr]|(.3)\lrcorner & \shcatExtAlg_0(S,M) \ar[d] \\
    \shcatHom (S_\bullet,R_\bullet) \ar[d] \ar[r] & \stb(S_\bullet, R_\bullet) \ar[d] \ar[r]^<<<p & \shcatExtAlg(S,M) \ar[d] \\
    \shcatHom_{k}(S_\bullet,R_\bullet) \ar[r] & \stb_{k}(S_\bullet, R_\bullet) \ar[r]^<<<p & \shcatExt_k(S,M)
  }
\end{equation*}
The objects on the leftmost column, as well as $\stb_0(S_\bullet, R_\bullet)$, consist of \emph{locally split} butterflies from $S_\bullet$ to $R_\bullet$ (cf.~\cite{ButterfliesI}).

\section{Butterflies as morphisms of ring-like stacks}
\label{sec:butt-morph-ring}
\numberwithin{equation}{section}%
In this section we prove our main result, that analogously to the case of group-like stacks, bimodule butterflies compute morphisms between stacks equipped with a ring-like structure.

Let $\sts$ and $\str$ be two ring-like stacks.  We denote by $\catHom(\sts,\str)$ the groupoid of (homo)morphisms from $\sts$ to $\str$, and by $\catHom_k(\sts,\str)$ the groupoid of morphisms of underlying Picard stacks. Similarly, we denote by $\shcatHom(\sts,\str)$ and $\shcatHom_k(\sts,\str)$ their respective stack analogs. Assume $\str$ and $\sts$ have presentations by crossed bimodules $R_\bullet : \xymatrix@1{M\ar[r]^\del & R}$ and $S_\bullet : \xymatrix@1{N\ar[r]^\del & S}$, respectively.
\begin{theorem}
  \label{thm:2}
  There are equivalences
  \begin{equation*}
    \catHom(\sts,\str) \iso \B(S_\bullet,R_\bullet)\quad\text{and}\quad
    \shcatHom (\sts,\str) \iso \stb (S_\bullet,R_\bullet).    
  \end{equation*}
\end{theorem}
This is the specialization to the context of ring-like stacks of the corresponding statements for Picard (or even just group-like) stacks proved in~\cite{ButterfliesI}. Indeed, forgetting the ring-like structures we get equivalences  
\begin{equation*}
  \catHom_k(\sts,\str) \iso \B_k(S_\bullet,R_\bullet)\quad\text{and}\quad
  \shcatHom_k (\sts,\str) \iso \stb_k (S_\bullet,R_\bullet).    
\end{equation*}
The necessary ingredient we will need is the construction of two mutually quasi-inverse functors $\Phi\colon \B_k(S_\bullet,R_\bullet)\to \Hom_k(\sts,\str)$ and $\Psi\colon \Hom_k(\sts,\str) \to \B_k(S_\bullet,R_\bullet)$.  We will recall some of the details of their definition from \loccit, then prove that they restrict to equivalences between $\catHom(\sts,\str)$ and $\B (S_\bullet, R_\bullet)$. Many of the ``moves'' in the new part of the proof would be a repeat of those already carried out in the original one, therefore we only sketch the main lines.

\numberwithin{equation}{subsection}%
\subsection{Recollections from~\protect{\cite[\S 4.3 and \S 4.4]{ButterfliesI}}}
\label{sec:recoll-bfl}

Throughout the proof we will use the equivalences $\sts\iso \tors(N,S)$ and $\str\iso \tors (M,R)$.  

Let $(Y,y)$ be an object of $\sts$.  Thus, $Y$ is an $N$-torsor equipped with an $N$-equivariant map $y\colon Y\to S$.  Let $E$ (by abuse of language) be a $k$-butterfly from $S_\bullet$ to $R_\bullet$.
First, define the $M$-torsor of local $N$-equivariant liftings of $y\colon Y\to S$ to $E$:
\begin{equation*}
  X \eqdef \shHom_N(Y,E)_y.
\end{equation*}
The $M$-action on $X$ takes the following form: if $\tilde y_0$ and $\tilde y_1$ are two different liftings defined over $U\in \Ob\s$, we have $\tilde y_0= \tilde y_1 + \imath (m)$, where, a priori, $m\colon Y\to M$. The $N$-equivariance of the lifts implies that $m$ is in fact $N$-\emph{invariant,} thus it only depends on the two lifts and not on the specific points of $Y$.  The torsor $X$ is equipped with the $M$-equivariant map  $x\colon X\to R$ defined by sending a (local) lift $\tilde y\colon Y\vert_U\to E$ to $\jmath\circ\tilde y$.  This is well defined: again, if $v,v+n$ are two points of $Y\vert_U$, with $n\in N\vert_U$, then
\begin{equation*}
  \tilde y (v+n) = \tilde y (v) + \kappa (n),
\end{equation*}
(see \loccit) so the post-composition with $\jmath$ does not depend on the specific point of $Y\vert_U$, but only on the lift itself.  Then, by definition, $\Phi(E)\colon \sts\to \str$ assigns to $(Y,y)$ the pair $(X,x)$ just defined.  It is clear that if $f\colon (Y',y')\to (Y,y)$ is a morphism of $(N,S)$-torsors, then we get a corresponding morphism $(X,x)\to (X',x')$.  Also, a morphism $\alpha\colon E\to E'$ of butterflies induces
\begin{equation*}
  \alpha_*(Y,y)\colon \shHom_N(Y,E)_y \lto \shHom_N(Y,E')_y,
\end{equation*}
and hence a morphism of functors
\begin{equation*}
  \alpha_* \colon \Phi(E) \Longrightarrow \Phi(E')\colon \sts\lto \str.
\end{equation*}
We refer to \loccit for details.

In the opposite direction, if $F\colon \sts\to \str$, then $E=\Psi (F)$ is the butterfly where:
\begin{equation*}
  E \eqdef S \times_{\str} R,
\end{equation*}
where the (stack) fiber product is computed with respect to the maps $\pi_\str\colon R\to \str$ and $F\circ \pi_\sts\colon S\to \str$.  Thus $E$ consists of triples $(s,\phi,r)$, where $s\in S$, $r\in R$, and $\phi\colon F(\pi_\sts(s))\to \pi_\str(r)$.  The maps $\pi\colon E\to S$ and $\jmath\colon E\to R$ are just the canonical projections to $S$ and $R$, respectively. The sequence $\xymatrix@1{M \ar[r]^\del & R \ar[r]^{\pi_\str} & \str}$ is homotopy exact, so its pullback along $F\circ \pi_\sts$ gives rise to the exact sequence $\xymatrix@1{M\ar[r]^\imath & E \ar[r]^\pi & S}$, the NE-SW diagonal of the butterfly.  The explicit form of the map $\imath\colon M\to E$ can be computed from the sequence: if $m\in M$, then we have:
\begin{equation*}
  \imath (m) = (0,\phi_m,\del m),
\end{equation*}
where $\phi_m\colon F(\pi_\sts(0)) = F(0_\sts) \to \pi_\str(\del m)$ is the composite  of the structural morphism $F(0_\sts) \to 0_\str$ with the (unique) isomorphism of torsors $\pi_\str(\del m)= (M,\del m) \isoto (M,0)=\pi_\str (0)= 0_\str$.  $\kappa$ can be defined along similar lines, bearing in mind that $\xymatrix@1{N\ar[r]^\del & S\ar[r]^{F\circ\pi_\sts} & \str}$ is only a complex, and so it will be its pullback along $\pi_\str$, giving rise to the NW-SE diagonal of the butterfly.  We refer to \loccit for further details on $\imath\colon M\to E$ and $\kappa\colon N\to E$ as well as the various functoriality properties.  

\subsection{Proof of Theorem~\ref{thm:2}}
\label{sec:proof}
  We show that, given a butterfly $E\in \B (S_\bullet,R_\bullet)$, the resulting morphism $\Phi(E)\colon \sts\to \str$ of Picard stacks is in fact ring-like by constructing isomorphisms
  \begin{equation}
    \label{eq:6}
    \Phi(E)(Y,y)\, \Phi(E)(Y',y') \lisoto \Phi(E)(YY',yy')
  \end{equation}
  satisfying the standard properties. If $(Y,y)$ and $(Y',y')$ are objects of $\sts$, we define the required isomorphism by sending two lifts $e\colon Y\to E$ and $e'\colon Y'\to E$ to the product $ee'$. This is well defined and compatible with the actions of $N$ on $Y$ and $Y'$, of $M$ on their images $X$ and $X'$, and with the definition of the product of torsors in sect.~\ref{sec:objects-products}. Indeed, for $v\in Y$, $v'\in Y'$, and $n,n'\in N$, we have:
  \begin{align*}
    e(v+n)\,e'(v'+n')
    & = e(v) \,e'(v') + e(v)\kappa(n') + \kappa (n)e'(v') + \kappa (n)\kappa(n') \\
    & = e(v) \,e'(v') + \kappa \bigl(
    \pi e (v)\, n' + n\,\pi e' (v') + n \, \pi\kappa(n') \bigr) \\
    & = e(v) \,e'(v') + \kappa  \bigl(
    y(v)\,n' + n\,y'(v') + n\,\del n' \bigr),
  \end{align*}
and we see the last line is just the equivariance of the lift $ee'$.  Similarly, for $m,m'\in M$, we have:
\begin{align*}
  ( e + \imath (m)) \, ( e' + \imath (m')) & = e\,e' + e\,\imath(m') + \imath(m)\,e' + \imath(m)\,\imath(m') \\
  & = e\, e' + \imath \bigl(
  \jmath\circ e\cdot m' + m\cdot \jmath\circ e' + m\cdot \jmath\circ \imath (m') \bigr) \\
  & = e\, e' + \imath \bigl(
  x(e)\, m' + m\, x'(e') + m\, \del m' \bigr).
\end{align*}
The verification that~\eqref{eq:6} is functorial and compatible with the associativity constraint follows the same steps as the proof in the group-like case of \loccit, and it is left to the reader.
  
Conversely, if $F\colon \sts\to \str$ is a morphism of ring-like stacks, then the resulting butterfly $E=\Psi(F)$ in $\B_k(S_\bullet, R_\bullet)$ actually satisfies the conditions in Definition~\ref{def:5}, with $E$ being equipped with a ring (or $k$-algebra) structure.

This is actually automatic, since $E=\Psi(F)=S\times_\str R$, so the pullback sequence
\begin{equation*}
  \xymatrix@1{%
    M \ar[r]^(.4)\imath & S\times_\str R \ar[r]^(.6)\pi & S
  }
\end{equation*}
comes naturally equipped with the structure of an algebra extension. Explicitly, the product in  $E$ reads:
\begin{equation*}
  (s,\phi,r)\,(s',\phi',r') \eqdef (ss', \phi\phi', rr'),
\end{equation*}
where $\phi\phi'$ stands for the composition:
\begin{equation*}
  F(\pi_\sts(s)\pi_\sts(s'))\iso F(\pi_\sts(s))F(\pi_\sts(s'))
  \xrightarrow{\phi\phi'} \pi_\str(r)\pi_\str(r')\iso \pi_\str(rr').
\end{equation*}
Associativity holds for the same reason it does for the sum operation in $E$.  Distributivity of the product with respect to the sum holds thanks to the fact that it (obviously) does in $S$ and $R$, and (weakly) in $\sts,\str$ and preserved by $F$.  For instance, for elements  $e_i=(s_i,\phi_i,r_i)$, for $i=0,1$ and $e=(s,\phi,r)$ of $E$, the equality $(e_0+e_1)e=e_0e + e_1e$ rests upon that of morphisms in $\str$
\begin{equation*}
  (\phi_0+\phi_1)\phi = \phi_0\phi + \phi_1\phi
\end{equation*}
(again, with shortened notation), which follows from the commutativity of structure diagrams as in~\cite[Definition 2.2]{MR2369166}.

It remains to prove that the butterfly satisfies the conditions~\eqref{eq:7}.
Let us pick just one of them, $\imath (m\, \jmath(e)) = \imath (m)\, e$.  Let $e=(s,\phi,r)$, as above. We have
\begin{equation*}
  \imath (m)\, e = (0,\phi_m,\del m)\, (s, \phi, r) = (0, \phi_m\phi, (\del m)r) = (0, \phi_m\phi, \del (mr)).
\end{equation*}
On the other hand, $\jmath (e) = r$, therefore
\begin{equation*}
  \imath (m\,\jmath(e)) = \imath (m\, r) = (0, \phi_{mr},\del( mr)).
\end{equation*}
Let $\eta$ be the structural isomorphism $F(0_\sts)\to 0_\str$. The commutativity of the diagram
\begin{equation*}
  \xymatrix{%
    F(0_\sts) \ar[d]_\eta \ar@{-}[r]^\sim \ar@(l,l)[dd]_{\phi_{mr}} & F(0_\sts) F(\pi_\sts(s))\ar[d]^{\eta\phi} \ar@(r,r)[dd]^{\phi_m\phi} \\
    0_\str \ar@{-}[r]^\sim & 0_\str \pi_\str(r) \\
    \pi_\str(\del (mr)) \ar[u]^{mr} \ar@{-}[r]^\sim & \pi_\str(\del m)\pi_\str(r) \ar[u]_{m\id}
  }
\end{equation*}
shows that, modulo the slight abuse of notation implied by omitting from it the standard isomorphisms, $\phi_{mr}=\phi_m\phi$, thereby implying the desired equality.  The remaining ones in~\eqref{eq:7} are treated similarly.\\ \hfill\qed

\section{Compositions of bimodule butterflies and the 2-stack of crossed bimodules}
\label{sec:comp-2stack-xbmod}

\subsection{Composition of butterflies}
\label{sec:composition}
Let $T_\bullet$, $S_\bullet$ and $R_\bullet$ be crossed bimodules.  We define a composition operation
\begin{equation*}
  \B (T_\bullet,S_\bullet) \times \B (S_\bullet, R_\bullet) \lto \B (T_\bullet, R_\bullet)
\end{equation*}
by restriction of the one for abelian sheaves defined in ref.\ in \cite{ButterfliesI}.  Consider the diagram
\begin{equation}
  \label{eq:9}
  \vcenter{%
    \xymatrix@-1pc{%
      P \ar[dd]_\del \ar[dr]^{\kappa'} && N \ar[dl]_{\imath'} \ar[dd]|\del \ar[dr]^\kappa && M \ar[dl]_\imath \ar[dd]^\del \\
      & F \ar[dl]_{\pi'} \ar[dr]^{\jmath'} & & E \ar[dl]_\pi \ar[dr]^\jmath \\
      T && S && R
    }} \quad =\quad
  \vcenter{%
    \xymatrix@-1pc{%
      P \ar[dd]_\del \ar[dr]^{\kappa''} && M \ar[dd]^\del \ar[dl]_{\imath''} \\
      & F\bigoplus^N_S E \ar[dl]_{\pi''} \ar[dr]^{\jmath''} \\
      T && R
    }
  }
\end{equation}
where $E$ is a butterfly from $S_\bullet$ to $R_\bullet$ and $F$ one from $T_\bullet$ to $S_\bullet$.  As an abelian sheaf, the object $F\oplus_S^NE$ is obtained as the cokernel of the monomorphism
\begin{equation}
  \label{eq:10}
  \xymatrix@1@C+1pc{N \ar[r]^(.4){(\imath',\kappa)} & F\bigoplus_S E}.
\end{equation}
It is proved in \loccit that the right hand side of~\eqref{eq:9} is in $\B_k(T_\bullet,R_\bullet)$, with $\pi''$ and $\jmath''$ being the obvious projections, whereas $\kappa''$ and $\imath''$ are induced by $(\kappa',0)$ and $(0,\imath)$, respectively.  In addition, $F\oplus_SE$ has an obvious algebra structure, and it is immediately seen that $N$ is an ideal via~\eqref{eq:10}. It is also easy to see the four morphisms $\kappa''$, $\imath''$, $\pi''$ and $\jmath''$ satisfy~\eqref{eq:7}, so the right hand side of~\eqref{eq:9} indeed is a bimodule butterfly. If $\beta\colon F'\to F$ and $\alpha\colon E'\to E$ are (iso)morphisms of butterfly, it is easily verified that there results a morphism
\begin{equation*}
  (\beta,\alpha)\colon F'\bigoplus^N_S E' \lto F\bigoplus^N_S E,
\end{equation*}
as a butterfly from $T_\bullet$ to $R_\bullet$.

This construction, analogously to the abelian sheaf case, can be sheafified over $\s$, so we obtain a composition law
\begin{equation*}
  \stb (T_\bullet,S_\bullet) \times \stb (S_\bullet, R_\bullet) \lto \stb (T_\bullet, R_\bullet).
\end{equation*}
In view of the equivalence of Theorem~\ref{thm:2}, we have
\begin{lemma}
  \label{lem:3}
  Let $\stt$, $\sts$ and $\str$ be the ring-like stacks corresponding to the above crossed bimodules. The composition law~\eqref{eq:9} is induced by that on ring-like stacks: 
  \begin{equation*}
    \shcatHom(\stt,\sts) \times \shcatHom(\sts,\str) \lto \shcatHom(\stt,\str)
  \end{equation*}
  via the equivalence of Theorem~\ref{thm:2}.
\end{lemma}
\begin{proof}
  Let $B\colon \stt\to \sts$ and $A\colon \sts\to \str$ be two morphisms. Let $F$ and $E$ be the corresponding butterflies.  We prove that the butterfly determined by $A\circ B$ is isomorphic to $F\oplus^N_S E$.

  From the proof of Theorem~\ref{thm:2} we have $E=S\oplus_\str\! R$ and $F=T\oplus_\sts\! S$.  Then
  \begin{equation*}
    F\bigoplus_S E = F\bigoplus_S \bigl( S\bigoplus_\str R \bigr)
    \iso F\bigoplus_\str R.
  \end{equation*}
  As a consequence, the morphism~\eqref{eq:10} equals
  \begin{equation*}
    \xymatrix@1@C+1pc{N \ar[r]^(.4){(\imath',0)} & F\bigoplus_\str R}
  \end{equation*}
  and its cokernel is therefore $T\oplus_\str\! R$, since
  \begin{equation*}
    \xymatrix@1{N \ar[r]^{\imath'} & F \ar[r]^{\pi'} & T}
  \end{equation*}
  is exact.  But $T\oplus_\str\! R$  is the center element of the butterfly determined by $A\circ B$, as wanted. Tracing the various steps shows that $F\oplus^N_SE\iso T\oplus_\str\!R$ is a ring isomorphism. We leave to the reader the tedious but straightforward verification that the above constructions are compatible with morphisms, namely that 2-morphisms $\alpha\colon A' \Rightarrow A \colon \sts\to \str$ and $\beta\colon B'\Rightarrow B\colon \stt\to \sts$ give rise to a morphism $F'\oplus^N_S E'\to F\oplus^N_SE$ of butterflies as above corresponding to $\alpha*\beta\colon A'\circ B' \Rightarrow A\circ B$.
\end{proof}

\subsection{The 2-stack of crossed bimodules}
\label{sec:2-stack-xbimod}

Let $\xbm (\s)$ the bicategory whose objects are crossed bimodules over $\s$.  The category (in fact, groupoid) of morphisms from the crossed bimodule $S_\bullet$ to $R_\bullet$ is the groupoid of butterflies $\B(S_\bullet,R_\bullet)$: since the composition~\eqref{eq:9} is obtained from the fiber product construction of the butterfly applied to the composite $\stt\to \sts\to \str$, the composition of butterflies is only associative up to isomorphism.  

As a consequence of Theorem~\ref{thm:1} and Theorem~\ref{thm:2} we have:
\begin{theorem}
  \label{thm:4}
  $\xbm (\s)$ is equivalent to the 2-category of ring-like stacks, \ie the 2-category of monoids in $\pic (\s)$. \qed
\end{theorem}
This is the specialization of a similar equivalence holding for the corresponding larger 2-categories of complexes and Picard stacks. More precisely, We have a (faithful) forgetful functor $\xbm (\s)\to \ch (\s)$, where the latter is  the bicategory of length 1-complexes of abelian sheaves over $\s$, equipped with butterflies $\B_k(-,-)$ as morphism groupoids.  Similarly, we can consider the whole 2-category of Picard stacks, $\pic (\s)$.  Then we have an equivalence of bicategories
\begin{equation*}
 \ch (\s) \iso \pic (\s). 
\end{equation*}
This equivalence can actually be sheafified over $\s$ to yield an equivalence of 2-stacks
\begin{equation*}
  \CH (\s) \lisoto \PIC (\s),
\end{equation*}
see \cite[Thm.\ 8.5.2 and Prop.\ 8.5.4]{ButterfliesI}. Note that $\CH (\s)$ is a 2-stack in a weaker sense, as it is fibered in bicategories.  The notable point is that complexes generally satisfy 2-descent with respect to butterflies, \ie weak morphisms, and not only strict ones.  The 2-descent arguments used in \loccit can be carried over the present situation. So we have:
\begin{theorem}
  \label{thm:3}
  The fibered bicategory defined by $U \to \xbm (\s/U)$, $U\in \Ob (\s)$, is a 2-stack $\XBM (\s)$.  Moreover, there is an equivalence with the 2-stack of monoids in $\PIC (\s)$. \qed
\end{theorem}

\section{The Shukla, Barr, André-Quillen cohomology class of a ring-like stack}
\label{sec:shukla}

One of the main points of refs.~\cite{MR2065328,MR2458668,MR2760344,MR3085798} is that (regular) $\mathit{Ann}$-categories with  given $\pi_0=A$ and $\pi_1=M$ are classified  up to equivalence by the third Shukla cohomology $\SH^3(A,M)$. In this last section, we briefly show that this is the case for the ring-like stacks of this paper as well, thereby providing another justification to our claim that ring-like stacks \emph{are} $\mathit{Ann}$-categories. We will be sketchy, as most of the material is well known.

\subsection{Cohomology or rings}
\label{sec:H-of-rings}

Let $k$ be a commutative ring, $A$ a $k$-algebra, and $X$ an $A$-bimodule. Let $\kalg$ be the category of $k$-algebras, and $\kalg/A$ that of $k$-algebras over $A$. After \cite{MR0257068} and \cite{MR0218424} (see also the accounts in \cite{MR1410176,MR0258917}), the third Shukla cohomology of $A$ with values in $X$ is
\begin{equation*}
  \SH^3(A/k,X) \iso \DerD^2_{k}(A,X),
\end{equation*}
where $\DerD^q_{k}(A,X)$ is the $q^{\mathrm{th}}$ derived functor of $A\rightsquigarrow \Der_k(A,X)$. More generally, for every object $C$ of $\kalg/A$, $\Der_k(C,X)$ consists of $k$-linear maps $D\colon C\to X$ such that $D(cc')=u(c)D(c')+D(c)u(c')$, where $u\colon C\to A$ is the structure map. We have
\begin{equation*}
  \Der_k(C,X) \iso \Hom_{A\otimes A^\op}((A\otimes A^\op)\otimes_{C\otimes C^\op} D_{C/k},X)\iso \Hom_{\kalg/B}(C,I_A(X)),
\end{equation*}
where
\begin{equation*}
  D_{C/k} = \ker (\mu\colon C\otimes_k C \lto C),
\end{equation*}
is the bimodule of differentials and $I_A(X)$ is the algebra of $X$-dual numbers over $A$.  Thus we have $\SH^3(A/k,X)\iso H^2(\Hom_{\kalg/A}(Y_\bullet,I_A(X)))$, where $Y_\bullet$ is a resolution of $B$ computed according one of the methods described in  \loccit

Alternatively, the associative algebra cohomology can be computed by way of an suitable model structure on the category of non-negatively graded chain differential graded algebras, as in \cite{MR2270566}. In this formulation, the element of $\SH^3(A/k,X)$ classifying the crossed extension
\begin{equation}
  \label{eq:12}
    \xymatrix@1{%
    0 \ar[r] & X\ar[r] & M \ar[r]^\del & R \ar[r] & A \ar[r] & 0},
\end{equation}
and by consequence that of the $\mathit{Ann}$-category $[M\to R]\sptilde$ associated to it, is computed from a cofibrant resolution of $A$ and the diagram
\begin{equation}
  \label{eq:13}
  \vcenter{%
    \xymatrix{%
      \cdots \ar[r] &
      A_3 \ar[r] \ar[d] & A_2\ar[r] \ar[d]^f & A_1 \ar[r] \ar[d] & A_0 \ar[r] \ar[d] & A \ar[r] \ar@{=}[d]& 0 \\
      & 0 \ar[r] & X\ar[r] & M \ar[r]^\del & R \ar[r] & A \ar[r] & 0
    }}
\end{equation}
The element $f$ in the above diagram is a derivation whose class is the element of $\SH^3(A/k,X)$ in question. 

\subsection{Functorial behavior}
\label{sec:functorial}

Consider a morphism
\begin{equation*}
  F\colon [N\lto S]\sptilde \lto [M\lto R]\sptilde
\end{equation*}
of $\mathit{Ann}$-categories. From Theorem~\ref{thm:2} we obtain a diagram with exact rows featuring the butterfly representation of $F$:
\begin{equation}
  \label{eq:14}
  \vcenter{\xymatrix@R-0.5pc{%
      0 \ar[r] & Y \ar[dd]^{\eta} \ar[r] & N \ar[dr]_\kappa \ar[rr]^\del && S \ar[r] & B \ar[dd]^{\xi} \ar[r] & 0\phantom{\,.} \\
      &         & & E \ar[ur]_\pi \ar[dr]^\jmath \\
      0 \ar[r] & X\ar[r] & M \ar[ur]^\imath \ar[rr]_\del && R \ar[r] & A \ar[r] & 0\,.
    }}
\end{equation}
From its properties we can readily obtain the morphisms $\xi$ and $\eta$; the former is in $\kalg$, the latter is a morphism of left $B\otimes B^\op$-modules, upon restricting scalars for $X$.

Let $y$ and $x$ denote the crossed extensions defined by the top and bottom rows of~\eqref{eq:14}, respectively, and denote by $[y]$ and $[x]$ the equivalence classes they determine.
\begin{proposition}
  \label{prop:1}
  The diagram~\eqref{eq:14} induces
  \begin{equation*}
    \xymatrix{%
      \SH^3(B/k,Y) \ar[r]^{\eta_*} & \SH^3(B/k,X)\phantom{\,,} \\
      & \SH^3(A/k,X) \ar[u]_{\xi^*}\,,}
  \end{equation*}
  such that $\eta_*([y]) = \xi^*([x])$.
\end{proposition}
\begin{proof}
  Let us use the notation $S_\bullet = [N\to S]$, $E_\bullet=[N\oplus_SE\to E]$, and $R_\bullet = [M\to R]$, where $\bullet=0,1$. Let us put primes for the corresponding DGAs, as in~\ref{ex:3}: thus $S'_\bullet = \cdots \to 0\to Y\to N\to S$, and so on.

By sect.~\ref{sec:fract-determ-butt} and Lemma~\ref{lem:2}, we have a quasi-isomorphism $p_\bullet\colon E_\bullet \to S_\bullet$, with $p_1=\tilde\pi$ and  $p_0=\pi$. In fact $p_\bullet$ induces the identity on $Y=\pi_1(E_\bullet)=\pi_1(S_\bullet)$ and $B=\pi_0(E_\bullet)=\pi_0(S_\bullet)$. Recall that $\pi$ is an epimorphism.

  Let $B_\bullet \colon \cdots \to B_2\to B_1 \to B_0$ be a cofibrant replacement of $B$ fitting in a diagram like~\eqref{eq:13}, and let $\delta\colon B_2\to Y$ represent the class of the crossed extension $y = Y\to [N\to S]\to B$ in $\SH^3(B/k,Y)$.

  The crossed bimodule $E_\bullet$ determines a crossed extension $y'=Y\to E_\bullet \to B$ whose class in $\SH^3(B/k,Y)$ is the same as that of $y$ since $p_\bullet\colon E_\bullet\to S_\bullet$ (and hence the corresponding map of DGAs) is a trivial fibration \cite{MR2270566}; indeed, the map $B_\bullet \to S'_\bullet$ lifts to $B_\bullet\to E'_\bullet$, and if $B'_\bullet\to E'_\bullet$ is a cofibrant replacement, so is the composition $B'_\bullet \to E'_\bullet\to S'_\bullet$, resulting in an isomorphims $B'_\bullet\iso B_\bullet$. Let $\delta'\colon B'_2\to Y$ be the derivation corresponding to $\delta$ which represents the class of $y'$.

  We can now proceed by using $E_\bullet$ in place of $S_\bullet$, and consider the morphism $f_\bullet\colon E_\bullet\to R_\bullet$ which also represents $F$. The following diagram has a filling
  \begin{equation*}
    \xymatrix{%
      B'_\bullet \ar@{->>}[d]^\wr \ar[dr] \ar@{.>}[r]^{\tilde f} & A_\bullet \ar@{->>}[d]^\wr \\
      E'_\bullet \ar[r]_f & R'_\bullet
    }
  \end{equation*}
  since $B'_\bullet$ is cofibrant. If $D\colon A_2\to X$ represents the class of the crossed extension $x=X \to [M\to R]\to A$ in $\SH^3(A/k,X)$ we have $D\circ \tilde f_2 = f_2\circ \delta$. Since $f_2$ is a bimodule map, it follows that
  \begin{equation*}
    [f_2\circ\delta]\in \SH^3(B/k,X).
  \end{equation*}
  Since $\tilde f$ covers $\xi\colon B\to A$, we obtain the statement.
\end{proof}
\begin{proposition}
  \label{prop:2}
  The diagram in Proposition~\ref{prop:1} depends only on the isomorphism class of $F$ (or of its representative $E$).
\end{proposition}
\begin{proof}
  From Definition~\ref{def:6} and Theorem~\ref{thm:2}, a natural isomorphism $a\colon F\Rightarrow F'\colon [N\to S]\sptilde \to [M\to R]\sptilde$ is represented by an isomorphism $\alpha\colon E\to E'$ fitting in the diagram~\eqref{eq:14}.
  Thus the statement follows from the calculations in the proof of Proposition~\ref{prop:1}.
\end{proof}

\subsection{The class of a ring-like stack}
\label{sec:class-ring-like}

Now let $(\s,k)$ be a ringed site, let $A$ be a $k$-algebra of $\T=\shs$, and let $X$ be an $A$-bimodule. Define $\derD^q_{k}(A,X)$ as the $q^{\mathrm{th}}$ derived functor of the composite functor
\begin{equation*}
  X\rightsquigarrow \Hom_{\shs_\ab}(\ZZ,\shDer_k(A,X)) \iso
  \Hom_{\shs_\ab}(\ZZ,\shHom(D_{A/k},X)),
\end{equation*}
where $\ZZ$ denotes the constant sheaf associated to $\ZZ[\pt]$, where $\pt$ denotes the terminal object of $\T=\shs$. (Note that for any abelian object $F$ of $\T$, $\Hom_{\shs_\ab}(\ZZ,F)\iso \Hom_{\shs}(\pt,F)$.) Also, we let $\shDerD^q_k(A,X)$ denote the sheaf associated to the presheaf $U\mapsto \DerD^q_k(A\rvert_U,X\rvert_U)$. By \cite{MR0257068,MR0223432}, $\DerD^q_k(A,X)$ can be computed as sheaf cohomology, hence we have
\begin{equation*}
  \Ext_{\shs_\ab}^p(\ZZ,\shDerD^q_k(A,X)) \Longrightarrow \derD_k^{p+q}(A,X).
\end{equation*}

Let $\str$ be a ring-like stack over $\s$, with $\pi_0(\str)=A$, $\pi_1(\str)=X$.
\begin{proposition}
  \label{prop:3}
  $\str$ determines a class in $\Hom_{\shs_\ab}(\ZZ,\shDerD^2_k(A,X))$.
\end{proposition}
\begin{proof}
Let $V_\bullet\to \pt$ be a hypercover of the terminal object of $\shs$ (or of $\s$ if it does have a terminal object). By Theorems~\ref{thm:2},~\ref{thm:3} and~\ref{thm:4}, we may represent $\str$ by descent along bimodule butterflies by way of a local equivalence $\str_0\iso \str\rvert_{V_0}$, where $\str_0\iso [M\to R]\sptilde$, and $\del\colon M\to R$ is a crossed bimodule of $(\s/V_0)\sptilde$, plus a morphism of ring-like stacks $F\colon d_0^*\str_0\isoto d_1^*\str_0$. We have a four-term exact sequence of the type of~\eqref{eq:12} over $V_0$,
\begin{equation*}
  \xymatrix@1{0 \ar[r] & \epsilon^* X \ar[r] & M  \ar[r]^\del & R \ar[r] & \epsilon^*A \ar[r] & 0\,,}
\end{equation*}
where $\epsilon\colon V_0\to \pt$, and over $V_1$ the morphism $F$ corresponds to a butterfly $E$ such that we have a diagram of the form~\eqref{eq:14}, namely
\begin{equation*}
  \xymatrix@R-0.4pc{%
    0 \ar[r] & (\epsilon d_0)^*X \ar[dd]^{\eta} \ar[r] & d_0^*M \ar[dr]_\kappa \ar[rr]^\del && d_0^*R \ar[r] & (\epsilon d_0)^*A \ar[dd]^{\xi} \ar[r] & 0\phantom{\,,} \\
    & & & E \ar[ur]_\pi \ar[dr]^\jmath \\
    0 \ar[r] & (\epsilon d_1)^*X\ar[r] & d_1^*M \ar[ur]^\imath \ar[rr]_\del && d_1^*R \ar[r] & (\epsilon d_1)^*A \ar[r] & 0\,.}
\end{equation*}
These data satisfy the obvious cocycle conditions in the 2-categorical sense (cf. \cite[\S 5.3]{ButterfliesI}): over $V_2$ we have a natural isomorphism $a\colon F_{01}\circ F_{12}\Rightarrow F_{02}$ satisfying the relation $a_{023}\circ a_{012} = a_{013}\circ a_{123}$ over $V_3$. (Recall the missing index convention: $R_0=(d_1d_2)^*R$,..., $F_{01}=d_2^*F$,..., $a_{012}=d_3^*a$, etc.) These conditions determine those satisfied by $E$ and its pullbacks; in particular, we have an isomorphism of butterflies with the morphism
\begin{equation*}
  \alpha \colon E_{12}\bigoplus^{M_{1}}_{R_{1}} E_{01}\lto E_{02}
\end{equation*}
satisfying the same coherence condition as $a$ over $V_3$.

Note that $\xi$ and $\eta$ in the above diagram equal the respective identity maps of $A$ and $X$ modulo the relations $d_0^*\epsilon^* \iso (\epsilon d_0)^* = (\epsilon d_1)^* \iso d_1^*\epsilon^*$, hence by Proposition~\ref{prop:1} we obtain an isomorphism
\begin{equation*}
  f \colon
  \xymatrix@1{%
    **[l] \SH^3((A\rvert_{V_1})/k,X\rvert_{V_1}) \ar@(dr,ur)[]_\wr }
\end{equation*}
carrying the class of the crossed extension $0\to X\to M_1\to R_1\to A\to 0$ to that of $0\to X\to M_0\to R_0\to A\to 0$.

By Proposition~\ref{prop:2}, $f$ satisfies the relation
\begin{equation*}
  f_{02} = f_{01}\circ f_{12}\,,
\end{equation*}
so that $f$ is a global section of $\shDerD^2_k(A,X)$, as wanted.
\end{proof}

\appendix

\section{On the isomorphism between Shukla and André-Quillen}
\label{sec:appendix}
\numberwithin{equation}{section}%

We reproduce an argument by T.~Pirashvili showing the isomorphism between Shukla cohomology and the cohomology of associative algebras as defined by Quillen.

Let $k$ be a commutative, unital ring, $A$ an associative $k$-algebra and $M$ an $A$-bimodule. It is well known that for a $k$-algebra $A$ one can modify the standard complex computing its Hochschild cohomology so that for the resulting complex $C^\bullet(A,M)$ we have:
\begin{equation*}
  H^i(C^\bullet(A,M)) =
  \begin{cases}
    \HH^{i+1}(A,M) & i>0, \\
    \Der_k(A,M) & i=0.
  \end{cases}
\end{equation*}
The extensions of $C^\bullet(A,M)$ to simplicial $k$-algebras or differential graded algebras (as usual assumed to be chain algebras, i.e.\ supported in negative cohomological degrees) are straightforward, with obvious modifications. For example, for a simplicial $k$-algebra $A_\bullet$ and a $\pi_0(A_\bullet)$-bimodule $M$, one applies the above construction degree-wise, to obtain a cosimplicial object in the category of cochain complexes:
\begin{math}
  [n] \mapsto C^\bullet(A_n,M).
\end{math}
Then consider the total complex $C^\bullet(A_\bullet,M)$ of this cosimplicial cochain complex. For differential graded algebras the situation is similar: if $A_\bullet$ is a chain algebra, $M$ is considered as a complex whose only nonzero term is placed in degree zero, so that in effect $M$ is an $H_0(A_\bullet)$-bimodule.

Also, if $A_\bullet$ is a simplicial $k$-algebra, denote by $A_\bullet\sptilde$ the chain complex whose differential is the alternating sum of all the face maps, and by $NA_\bullet$ its normalization.

Let $A$ be a $k$-algebra. A simplicial $k$-algebra $P_\bullet$ such that each $P_i$ is a projective $k$-module and
\begin{equation*}
  \pi_i(P_\bullet) = 
  \begin{cases}
    0 & i>0,\\
    A & i=0
  \end{cases}
\end{equation*}
is called a flat simplicial resolution of $A$. Cofibrant replacements of $A$ are taken in the model category of simplicial $k$-algebras of~\cite{MR0257068}. Any cofibrant replacement is a flat simplicial resolution.
\begin{lemma}\hfill
  \begin{enumerate}
  \item Let $f\colon F_\bullet\to P_\bullet$ be a weak equivalence of flat simplicial resolutions of $A$. Then the resulting morphism
    \begin{math}
      C^\bullet(P_\bullet,M) \lto C^\bullet(F_\bullet,M)
    \end{math}
    is a weak equivalence.
  \item If $P_\bullet$ is a flat simplicial resolution of $A$, then we have an isomorphism $\DerD_k^\bullet(A,M)\iso H^\bullet(C^\bullet(P_\bullet,M))$, where $\DerD_k^\bullet(A,M)$ is the associative algebra cohomology defined in \loccit
  \end{enumerate}
\end{lemma}
\begin{proof}
  The first point follows from Kunneth theorem. For the second, we can assume that $P_\bullet$ is a cofibrant replacement of $A$. In this case each $P_i$ is free, so $H^j(C^\bullet(P_i,M))=0$ for $j>0$. Then the spectral sequence determined by the bicomplex $C^\bullet(P_\bullet,M)$ degenerates, yielding the required isomorphism.
\end{proof}
For the analogous result for differential graded algebras, see \cite{MR2270566}.

\begin{claim}
  There is an isomorphism $\SH^{i+1}(A/k,M)\iso \DerD^i_k(A,M)$.
\end{claim}
\begin{proof}
  Let $P_\bullet$ be a flat simplicial resolution of $A$. Then the normalization $NP_\bullet$ is a flat resolution of $A$, and so it is enough to show that $C^\bullet(P_\bullet,M)$ and $C^\bullet(NP_\bullet,M)$ have isomorphic cohomology.

  Now, there is a weak equivalence $NP_\bullet \to P_\bullet\sptilde$ which, by the Eilenberg-Zilber theorem, gives a weak equivalence $NP_\bullet\otimes\dots \otimes NP_\bullet\to P_\bullet\sptilde\otimes \dots \otimes P_\bullet\sptilde$ between the $n$-fold tensor products. It follows that $C^\bullet(P_\bullet,M)\to C^\bullet(NP_\bullet,M)$ is a weak equivalence.
\end{proof}

\phantomsection  
\addcontentsline{toc}{section}{References}  
\printbibliography
\end{document}